\documentclass[11pt, letterpaper,final]{amsart}
\usepackage{amsfonts,amsmath, amsthm, amssymb, latexsym, epsfig}
\usepackage[english]{babel}
\usepackage[utf8]{inputenc}
\usepackage[all]{xy}
\usepackage{xspace}
\usepackage{comment}
\usepackage{setspace}
\usepackage{enumerate}
\usepackage[mathscr]{eucal}
\usepackage{mathrsfs}
\usepackage{pigpen}

\usepackage[notcite,notref]{showkeys}

	\advance \topmargin by -\headheight
	\advance \topmargin by -\headsep

	\evensidemargin \oddsidemargin
	\newcommand{\Ext}{\ensuremath{\operatorname{Ext}}}
	
	\newcommand{\multialg}[1]{\mathscr{M}(#1)\xspace}
	\newcommand{\corona}[1]{\mathscr{Q}(#1)\xspace}

	\newcommand{\Ad}{\ensuremath{\operatorname{Ad}}\xspace}

	\newcommand{\Hom}{\mathrm{Hom}}
	\newcommand{\Aut}{\mathrm{Aut}}
	\newcommand{\coker}{\mathrm{coker}}
	\newcommand{\Ksixnu}{K_{\mathrm{six}}}
	\newcommand{\Ksix}{K_{\mathrm{six}}^{\mathrm{u}}}
	\newcommand{\Ksixpos}{K_{\mathrm{six}}^{+,\mathrm{u}}}

	\newcommand{\Knew}{\widetilde{K}_{\mathrm{six}}^{+, \Sigma}}
	\newcommand{\id}{\mathrm{id}}
	\newcommand{\us}{\mathrm{us}}
	\newcommand{\uw}{\mathrm{uw}}
	\newcommand{\ev}{\mathrm{ev}}
	\newcommand{\w}{\mathrm{w}}
	\newcommand{\s}{\mathrm{s}}

	\newcommand{\im}{\mathrm{im}\,}

	\newcommand{\kk}{\mathrm{KK}}

	\theoremstyle{plain}
	\newtheorem{thm}{Theorem}[section]
	\newtheorem{lemma}[thm]{Lemma}
	\newtheorem{theorem}[thm]{Theorem}
	\newtheorem{proposition}[thm]{Proposition}
	\newtheorem{corollary}[thm]{Corollary}

	\newtheorem{theoremintro}{Theorem}

	\theoremstyle{definition}
	\newtheorem{definition}[thm]{Definition}
	
	\newtheorem{remark}[thm]{Remark}
	\newtheorem{notation}[thm]{Notation}

	\numberwithin{equation}{section}
	\numberwithin{figure}{section}

\begin{document}
	\title[The unital $\mathrm{Ext}$-groups and classification of $C^*$-algebras]{The unital $\mathbf{Ext}$-groups and classification of $C^*$-algebras}
	\author[J. Gabe]{James Gabe}
	\address[J. Gabe]{School of Mathematics and Statistics\\University of Glasgow\\University Place\\Glasgow\\G12 8SQ\\Scotland}
        \email{jamiegabe123@hotmail.com}
	\author[E. Ruiz]{Efren Ruiz}
	\address[E. Ruiz]{Department of Mathematics\\University of Hawaii,
	Hilo\\200 W. Kawili St.\\
	Hilo, Hawaii\\
	96720-4091 USA}
	\email{ruize@hawaii.edu}
	\thanks{The first named author was funded by the Carlsberg Foundation through an Internationalisation Fellowship. Parts of the paper were completed while the first named author was a PhD student, at which time he was funded by Danish National Research Foundation through the Centre for Symmetry and Deformation (DNRF92). The second named author was supported by the Simons Foundation \#567380 to Ruiz.}

\begin{abstract}
The semigroups of unital extensions of separable $C^\ast$-algebras come in two flavours: a strong and a weak version. By the unital $\Ext$-groups, we mean the groups of invertible elements in these semigroups. We use the unital $\Ext$-groups to obtain $K$-theoretic classification of both unital and non-unital extensions of $C^\ast$-algebras, and in particular we obtain a complete $K$-theoretic classification of full extensions of UCT Kirchberg algebras by stable AF algebras.  
\end{abstract}

\maketitle

\section{Introduction}

Elliott's programme of classifying nuclear $C^\ast$-algebras has seen great recent success in the case of finite, simple $C^\ast$-algebras due to the work of many hands, most prominently by work of Elliott, Gong, Lin, and Niu \cite{GongLinNiu-classZ-stable}, \cite{ElliottGongLinNiu-classfindec}, as well as the Quasidiagonality Theorem of Tikuisis, White, and Winter \cite{TikuisisWhiteWinter-qdnuc}. This crowning achievement together with the ground breaking Kirchberg--Phillips classification of purely infinite, simple $C^\ast$-algebras \cite{Kirchberg-simple}, \cite{Phillips-classification} completes the classification of separable, unital, simple $C^\ast$-algebras with finite nuclear dimension which satisfy the universal coefficient theorem (UCT). 

The main focus of this paper is the classification of non-simple $C^\ast$-algebras. The non-simple classification is especially convoluted due to the lack of a dichotomy between the purely infinite and the stably finite case. A rich class of non-simple $C^\ast$-algebras failing this dichotomy is the class of graph $C^\ast$-algebras. Great progress was made recently in \cite{ERRS-classunitalgraph}, where all unital graph $C^\ast$-algebras were classified by a $K$-theoretic invariant. 

The classification of unital graph $C^\ast$-algebras was an internal classification result, in the sense that it can only be used to compare objects which are already known to be unital graph $C^\ast$-algebras. The lack of external classification prevents the result from being applicable in the study of permanence properties for the class of graph $C^\ast$-algebras. For instance, it is an open problem whether extensions of graph $C^\ast$-algebras are again graph $C^\ast$-algebras, subordinate to $K$-theoretic obstructions. The main results of this paper will be used to solve this question for extensions of simple graph $C^\ast$-algebras in \cite{EGKRT-extgraph}.
 
The focal point for us is the classification of extensions of classifiable $C^\ast$-algebras.
In seminal work of Rørdam \cite{Rordam-classsixterm}, a  Weyl--von Neumann--Voiculescu type absorption theorem of Kirchberg was applied to obtain classification of extensions of non-unital UCT Kirchberg algebras.\footnote{A \emph{UCT Kirchberg algebra} is a separable, nuclear, simple, purely infinite $C^\ast$-algebra satisfying the universal coefficient theorem in $\kk$-theory.} This absorption theorem was generalised by Elliott and Kucerovsky \cite{ElliottKucerovsky-extensions}, thus making the techniques of Rørdam applicable for much more general classification results, as explored by Eilers, Restorff, and Ruiz in \cite{EilersRestorffRuiz-classext}. 

These methods relied heavily on the \emph{non-unital} $\Ext$-group, which is known to be isomorphic to Kasparov's group $\kk^1$.
It is not hard to observe that similar methods should apply to \emph{unital} extensions if one applies the strong unital $\Ext$-group $\Ext_\us^{-1}(\mathfrak A, \mathfrak B)$ instead. One difficulty in working with the strong $\Ext$-group is that it is even more sensitive than $\kk$-theory. For instance, let $u\in \mathfrak A$ be a unitary. In contrast to $\kk$-theory where $\kk(\Ad u) = \kk(\id_A)$, the automorphism on $\Ext_{\us}^{-1}(\mathfrak A, \mathfrak B)$ induced by $\Ad u$ is \emph{not necessarily} the identity map. The same phenomena will never happen for the weak $\Ext$-group $\Ext_{\uw}^{-1}(\mathfrak A, \mathfrak B)$ as it embeds naturally as a subgroup of $\kk^1(\mathfrak A, \mathfrak B)$.

In \cite[Theorem~3.9]{EilersRestorffRuiz-classext}, all full extensions of \emph{non-unital} UCT Kirchberg algebras by stable AF algebras are classified by their six-term exact sequences in $K$-theory (with order in $K_0$ of the ideal). We will complete the classification of such extensions obtaining classification in the case where the UCT Kirchberg algebra is unital. This will be divided into two cases: one where the extension algebra is unital, and one where it is non-unital.

In the case of unital extensions, the invariant will be $\Ksixpos$ which is the six-term exact sequence in $K$-theory together with order and position of the unit in the $K_0$-groups. The classification is as follows.

\begin{theoremintro}\label{t:classunital}
Let $\mathfrak e_i : 0 \to \mathfrak B_i \to \mathfrak E_i \to \mathfrak A_i \to 0$ be unital extensions of $C^\ast$-algebras for $i=1,2$ such that $\mathfrak A_1$ and $\mathfrak A_2$ are UCT Kirchberg algebras, and $\mathfrak B_1$ and $\mathfrak B_2$ are stable AF algebras. Then $\mathfrak E_1\cong \mathfrak E_2$ if and only if $\Ksixpos(\mathfrak e_1) \cong \Ksixpos(\mathfrak e_2)$.
\end{theoremintro}

Next we turn our attention to non-unital extensions with unital quotients. A unital extension as considered above will always be full, as the Busby map is unital and the quotient is simple. For non-unital extensions it is in general much harder to determine whether they are full or not. However, when mixing sufficient amounts of finiteness and infiniteness, it turns out that fullness is a very natural criterion, witnessed by the existence of a properly infinite, full projection in the extension algebra, see Theorem \ref{t:fullext}.

In \cite{Gabe-nonunitalext}, examples were given of non-isomorphic full extensions of the Cuntz algebra $\mathcal O_2$ by the stabilised CAR algebra $M_{2^\infty} \otimes \mathbb K$, which had isomorphic six-term exact sequences in $K$-theory with order, scales and units in the $K_0$-groups. This means that one needs a finer invariant to classify non-unital extensions when the quotient is unital. 

For this purpose, we introduce an invariant $\Knew$ which includes the usual six-term exact sequence of the extension $0 \to \mathfrak B \to \mathfrak E \xrightarrow \pi \mathfrak A \to 0$, together with the $K$-theory of the extension $0\to \mathfrak B \to \pi^{-1}(\mathbb C 1_{\mathfrak A}) \to \mathbb C \to 0$. We refer the reader to Section \ref{s:nonunital} for more details.

\begin{theoremintro}\label{t:classnonunital}
Let $\mathfrak e_i : 0 \to \mathfrak B_i \to \mathfrak E_i \to \mathfrak A_i \to 0$ be full extensions of $C^\ast$-algebras for $i=1,2$ such that $\mathfrak A_1$ and $\mathfrak A_2$ are unital UCT Kirchberg algebras, $\mathfrak B_1$ and $\mathfrak B_2$ are stable AF algebras. Then $\mathfrak E_1 \cong \mathfrak E_2$ if and only if $\Knew(\mathfrak e_1) \cong \Knew(\mathfrak e_2)$.
\end{theoremintro}
 
In the paper \cite{EGKRT-extgraph} we will compute the range of the invariant $\Knew$ for graph $C^\ast$-algebras with exactly one non-trivial ideal and for which the non-trivial quotient is unital. This will be used to show that an extension of simple graph $C^\ast$-algebras is again a graph $C^\ast$-algebra, provided there are no $K$-theoretic obstructions.

%%%%%%%%%%%%%%%%%%%%%%%%%%%%%%%%%%%%%%%%%%%%%%%%%%%%%%%%
%%%%%%%%%%%%%%%%%%%%%%%%%%%%%%%%%%%%%%%%%%%%%%%%%%%%%%%%

\section{Extensions of $C^*$-algebras}

In this section we recall some well-known definitions and results about extensions of $C^\ast$-algebras. More details can be found in \cite[Chapter VII]{Blackadar-book-K-theory}.

For a $C^*$-algebra $\mathfrak{B}$, we will denote the multiplier algebra by $\multialg{\mathfrak{B}}$, the corona algebra $\multialg{ \mathfrak{B} } / \mathfrak{B}$ by $\corona{\mathfrak{B}}$, and the canonical $\ast$-epimorphism from $\multialg{ \mathfrak{B} }$ to $\corona{ \mathfrak{B} }$ by $\pi_\mathfrak{B}$.  

Let $\mathfrak{A}$ and $\mathfrak{B}$ be $C^*$-algebras.  An \emph{extension of $\mathfrak A$ by $\mathfrak B$} is a short exact sequence 
\[
\mathfrak e : 0 \to \mathfrak{B} \xrightarrow{\iota} \mathfrak{E} \xrightarrow{\pi} \mathfrak{A} \to 0
\]
of $C^\ast$-algebras. Often we just refer to such a short exact sequence above, as an extension of $C^\ast$-algebras. At times we identify $\mathfrak B$ with its image $\iota(\mathfrak B)$ in $\mathfrak E$, which is a two-sided, closed ideal, and at times we identify $\mathfrak A$ with the quotient $\mathfrak E/\iota(\mathfrak B)$.

To any extension of $C^\ast$-algebras as above, there are induced $\ast$-homomorphisms $\sigma \colon \mathfrak E \to \multialg{\mathfrak B}$ and $\tau \colon \mathfrak A \to \corona{\mathfrak B}$, the latter of these called the \emph{Busby map} (or Busby invariant) of $\mathfrak e$. We sometimes refer to arbitrary $\ast$-homomorphisms $\mathfrak A \to \corona{\mathfrak B}$ as Busby maps.

An extension can be recovered up to canonical isomorphism of extensions by its Busby map $\tau$, as the extension
\[
0 \to \mathfrak B \to \mathfrak A \oplus_{\tau, \pi_{\mathfrak B}} \multialg{\mathfrak B} \to \mathfrak A \to 0
\]
where
\[
\mathfrak A \oplus_{\tau, \pi_{\mathfrak B}} \multialg{\mathfrak B} = \{ a\oplus m \in \mathfrak A \oplus \multialg{\mathfrak B} : \tau(a) = \pi_{\mathfrak B}(m)\}
\]
is the pull-back of $\tau$ and $\pi_{\mathfrak B}$. 

An extension is \emph{unital} if the extension algebra is unital, or equivalently, if the Busby map is a unital $\ast$-homomorphism.

A (unital) extension $\mathfrak e : 0\to \mathfrak B \to \mathfrak E \xrightarrow \pi \mathfrak A \to 0$ is called \emph{trivial} (or split) if there is a (unital) $\ast$-homomorphism $\rho \colon \mathfrak A \to \mathfrak E$ such that $\pi \circ \rho = \id_\mathfrak{A}$.\footnote{Note that a unital extension being trivial is slightly different from an extension -- which happens to be unital -- being trivial. In fact, the first requires $\rho(1_{\mathfrak A}) = 1_{\mathfrak E}$ which the other does not, and in general these two notions are different.} The extension $\mathfrak e$ is called \emph{semi-split} if there is a (unital) completely positive map $\eta \colon \mathfrak A \to \mathfrak E$ such that $\pi \circ \rho = \id_\mathfrak{A}$.

Let $\mathfrak e_i : 0 \to \mathfrak B \to \mathfrak E_i \to \mathfrak A \to 0$ be extensions of $C^\ast$-algebras with Busby maps $\tau_i$ for $i=1,2$. We say that $\mathfrak e_1$ and $\mathfrak e_2$ are \emph{strongly unitarily equivalent}, written $\mathfrak e_1 \sim_\s \mathfrak e_2$, if there exists a unitary $u\in \multialg{\mathfrak B}$ such that $\Ad \pi_{\mathfrak B}(u) \circ \tau_1 = \tau_2$.

By identifying $\mathfrak E_i$ with $\mathfrak A \oplus_{\tau_i , \pi_{\mathfrak B}} \multialg{\mathfrak B}$, we obtain the following commutative diagram
\begin{equation}\label{eq:suediag}
\xymatrix{
0 \ar[r] & \mathfrak B \ar[d]^{\Ad u}_\cong \ar[rr] && \mathfrak E_1 \ar[d]^{\Ad(1_{\mathfrak A} \oplus u)}_{(\cong)} \ar[rr] && \mathfrak A \ar@{=}[d] \ar[r] & 0 \\
0 \ar[r] & \mathfrak B \ar[rr] && \mathfrak E_2 \ar[rr] && \mathfrak A \ar[r] & 0
}
\end{equation}
with exact rows, which shows that $\Ad(1_{\mathfrak A} \oplus u) \colon \mathfrak E_1 \xrightarrow \cong \mathfrak E_2$ is an isomorphism by the five lemma.

Similarly, $\mathfrak e_1$ and $\mathfrak e_2$ are \emph{weakly unitary equivalent}, written $\mathfrak e_1 \sim_\w \mathfrak e_2$, if there exists a unitary $u\in \corona{\mathfrak B}$ such that $\Ad u \circ \tau_1 = \tau_2$. 

In contrast to strong unitary equivalence, we cannot in general conclude that the extension algebras $\mathfrak E_1$ and $\mathfrak E_2$ are isomorphic from  weak unitary equivalence.

\begin{remark}[Cuntz sum]
If $\mathfrak B$ is a stable $C^\ast$-algebra, then there are isometries $s_1,s_2\in \multialg{\mathfrak B}$ such that $s_1s_1^\ast + s_2 s_2^\ast = 1$. Such a pair $s_1,s_2$ are called \emph{$\mathcal O_2$-isometries}.

If $\mathfrak e_i : 0 \to \mathfrak B \to \mathfrak E_i \to \mathfrak A \to 0$ are extensions with Busby maps $\tau_i$ for $i=1,2$, then we let $\mathfrak e_1 \oplus_{s_1,s_2} \mathfrak e_2$ denote the extension of $\mathfrak A$ by $\mathfrak B$ with Busby map $\tau$ given by
\[
\tau(a) = \pi_{\mathfrak B}(s_1) \tau_1(a) \pi_{\mathfrak B}(s_1)^\ast + \pi_{\mathfrak B}(s_2) \tau_2(a) \pi_{\mathfrak B}(s_2)^\ast
\]
for $a\in A$. This construction is independent of the choice of $s_1$ and $s_2$ up to strong unitary equivalence, and thus we often write $\mathfrak e_1 \oplus \mathfrak e_2$, when we only care about the extension up to $\sim_\s$.
\end{remark}

\begin{definition}
Let $\mathfrak A$ and $\mathfrak B$ be separable $C^\ast$-algebras with $\mathfrak B$ stable. We let
\begin{itemize}
\item $\Ext(\mathfrak A, \mathfrak B)$ denote the semigroup of extensions of $\mathfrak A$ by $\mathfrak B$ modulo the relation defined by $[\mathfrak e_1] = [\mathfrak e_2]$ if and only if there exist trivial extensions $\mathfrak f_1,\mathfrak f_2$ of $\mathfrak A$ by $\mathfrak B$ such that
\[
\mathfrak e_1 \oplus \mathfrak f_1 \sim_\w \mathfrak e_2 \oplus \mathfrak f_2,
\]
or equivalently, there exist trivial extensions $\mathfrak f_1' , \mathfrak f_2'$ of $\mathfrak A$ by $\mathfrak B$ (which can be taken as $\mathfrak f_i' = \mathfrak f_i \oplus 0$) such that
\[
\mathfrak e_1 \oplus \mathfrak f_1' \sim_\s \mathfrak e_2 \oplus \mathfrak f_2'.
\]
\end{itemize}
Moreover, if $\mathfrak A$ is unital then we let
\begin{itemize}
\item $\Ext_\us(\mathfrak A, \mathfrak B)$ denote the semigroup of \emph{unital} extensions of $\mathfrak A$ by $\mathfrak B$ modulo the relation defined by $[\mathfrak e_1]_\s = [\mathfrak e_2]_\s$ if and only if there exist trivial, \emph{unital} extensions $\mathfrak f_1,\mathfrak f_2$ of $\mathfrak A$ by $\mathfrak B$ such that
\[
\mathfrak e_1 \oplus \mathfrak f_1 \sim_\s \mathfrak e_2 \oplus \mathfrak f_2.
\]
\item $\Ext_\uw(\mathfrak A, \mathfrak B)$ denote the semigroup of \emph{unital} extensions of $\mathfrak A$ by $\mathfrak B$ modulo the relation defined by $[\mathfrak e_1]_\w = [\mathfrak e_2]_\w$ if and only if there exist trivial, \emph{unital} extensions $\mathfrak f_1,\mathfrak f_2$ of $\mathfrak A$ by $\mathfrak B$ such that
\[
\mathfrak e_1 \oplus \mathfrak f_1 \sim_\w \mathfrak e_2 \oplus \mathfrak f_2.
\]
\end{itemize}
If $\mathfrak B$ is not stable, we define $\Ext_{(\us/\uw)}(\mathfrak A, \mathfrak B) := \Ext_{(\us/\uw)}(\mathfrak A, \mathfrak B \otimes \mathbb K)$.
\end{definition}

It is not hard to show that $\Ext_{(\us/\uw)}(\mathfrak A, \mathfrak B)$ is an abelian monoid, and that any trivial (unital) extension induces the zero element.
Hence the following makes sense.

\begin{definition}
Let $\Ext^{-1}(\mathfrak A,\mathfrak B)$, $\Ext^{-1}_\us(\mathfrak A, \mathfrak B)$ and $\Ext^{-1}_\uw(\mathfrak A , \mathfrak B)$ denote the subsemigroups of $\Ext(\mathfrak A,\mathfrak B)$, $\Ext_\us(\mathfrak A, \mathfrak B)$ and $\Ext_\uw(\mathfrak A , \mathfrak B)$ respectively (whenever these make sense), of elements which have an additive inverse. These subsets are abelian groups.
\end{definition}

\begin{remark}[Semisplit extensions]
Let $\mathfrak A$ and $\mathfrak B$ be separable $C^\ast$-algebras with $\mathfrak B$ stable (and $\mathfrak A$ unital). As in \cite[Section 15.7]{Blackadar-book-K-theory} it follows that a (unital) extension of $\mathfrak A$ by $\mathfrak B$ induces an element in $\Ext(\mathfrak A, \mathfrak B)$ (resp.~in either $\Ext_\us(\mathfrak A, \mathfrak B)$ or $\Ext_\uw(\mathfrak A, \mathfrak B)$) which has an additive inverse, if and only if the extension is semisplit.

In particular, if $\mathfrak A$ is nuclear it follows from the Choi--Effros Lifting Theorem \cite{ChoiEffros-lifting} that
\[
\Ext^{-1}(\mathfrak A,\mathfrak B) = \Ext(\mathfrak A, \mathfrak B), \quad \Ext^{-1}_\us(\mathfrak A,\mathfrak B) = \Ext_\us(\mathfrak A,\mathfrak B), \quad \Ext^{-1}_\uw(\mathfrak A,\mathfrak B) = \Ext_\uw(\mathfrak A,\mathfrak B).
\]
\end{remark}

\begin{definition}[Pull-back and push-out extensions]\label{d:pbpo}
Let $\mathfrak{e} : 0 \to \mathfrak B \to \mathfrak E \to \mathfrak A \to 0$ be an extension of $C^\ast$-algebras with Busby map $\tau$, and let $\alpha \colon \mathfrak C \to \mathfrak A$ be a $\ast$-homomorphism.  The \emph{pull-back extension} $\mathfrak e \cdot \alpha$ is the extension of $\mathfrak C$ by $\mathfrak B$ with Busby map $\tau \circ \alpha$.

If $\beta \colon \mathfrak B \to \mathfrak D$ is a non-degenerate $\ast$-homo\-morphism\footnote{A $\ast$-homomorphism $\beta \colon \mathfrak B \to \mathfrak D$ is \emph{non-degenerate} (or proper) if $\overline{\beta(\mathfrak B) \mathfrak D} = \mathfrak D$.} there is an induced unital $\ast$-homo\-morphism $\overline \beta \colon \corona{\mathfrak B} \to \corona{\mathfrak D}$.\footnote{In fact, $\beta$ induces a unital $\ast$-homomorphism $\multialg{\beta} \colon \multialg{\mathfrak B} \to \multialg{\mathfrak D}$ by $\multialg{\beta}(m) (\beta(b)d) := \beta(mb) d$ for $m \in \multialg{\mathfrak B}$, $b\in \mathfrak B$ and $d\in \mathfrak D$. This $\ast$-homomorphism descends to a unital $\ast$-homomorphism $\overline \beta \colon \corona{\mathfrak B} \to \corona{\mathfrak D}$.} The \emph{push-out extension} $\beta \cdot \mathfrak e$ is the extension of $\mathfrak A$ by $\mathfrak D$ with Busby map $\overline \beta \circ \tau$.

If $\eta \colon \corona{\mathfrak B} \to \corona{\mathfrak D}$ is a $\ast$-homomorphism, then we let $\eta \cdot \mathfrak e$ denote the extension of $\mathfrak A$ by $\mathfrak D$ with Busby map $\eta \circ \tau$. In particular, with $\beta$ as above, we have $\beta \cdot \mathfrak e = \overline \beta \cdot \mathfrak e$.
\end{definition}

With the notation as above, the push-out and pull-back extensions fit into the following commutative diagram with exact rows
\[
\xymatrix{
\mathfrak e \cdot \alpha : & 0 \ar[r] & \mathfrak B \ar@{=}[d] \ar[r] & \mathfrak E_\alpha \ar[d] \ar[r] & \mathfrak C \ar[d]^\alpha \ar[r] & 0 \\
\mathfrak e : & 0 \ar[r] & \mathfrak B \ar[d]^\beta \ar[r] & \mathfrak E \ar[d] \ar[r] & \mathfrak A \ar[r] \ar@{=}[d] & 0 \\
\beta \cdot \mathfrak e : & 0 \ar[r] & \mathfrak D \ar[r] & \mathfrak E_\beta \ar[r] & \mathfrak A \ar[r] & 0.
}
\]
The top two rows form a pull-back diagram and the bottom two rows form a push-out diagram.

\begin{remark}[Functoriality]
The pull-back/push-out constructions of extensions turn $\Ext_{(\us/\uw)}(\mathfrak A, \mathfrak B)$ into a bifunctor with respect to (unital) $\ast$-homomorphisms in the first variable, and non-degenerate $\ast$-homomorphisms in the second variable.

A fair warning: while any unital $\ast$-homomorphism $\eta \colon \corona{\mathfrak B} \to \corona{\mathfrak D}$ induces a map $\mathfrak e \mapsto \eta \cdot \mathfrak e$ which preserves $\sim_\w$ (and $\sim_\s$ if $\mathfrak B$ is stable\footnote{In fact, if $\mathfrak B$ is stable then the unitary group $\mathcal U(\multialg{\mathfrak B})$ is connected, and thus a unitary $u\in \corona{\mathfrak B}$ lifts to a unitary in $\multialg{\mathfrak B}$ exactly when $u\in  \mathcal U_0(\corona{\mathfrak B})$, i.e.~the connected component of $1_\corona{\mathfrak B}$ in the unitary group. As $\eta(\mathcal U_0(\corona{\mathfrak B})) \subseteq \mathcal U_0(\corona{\mathfrak D})$, and as every unitary in $\mathcal U_0(\corona{\mathfrak D})$ lifts to a unitary in $\multialg{\mathfrak D}$, it easily follows that $\eta$ preserves strong unitary equivalence classes of extensions.}), it does in general not preserve Cuntz sums. This construction will be crucial in Remark \ref{r:ex} where we define $\mathfrak e_{[u]} = \Ad u \cdot \mathfrak e_0$ for a unitary $\mathcal U(\corona{\mathfrak B})$ and a trivial unital extension $\mathfrak e_0$.
\end{remark}

The following is a celebrated result of Kasparov \cite{Kasparov-KKExt}.

\begin{theorem}[{\cite{Kasparov-KKExt}}]\label{t:KKExt}
If $\mathfrak A$ and $\mathfrak B$ are separable $C^\ast$-algebras, then $\Ext^{-1}(\mathfrak A, \mathfrak B)$ is naturally isomorphic to Kasparov's group $\kk^1(\mathfrak A, \mathfrak B)$.
\end{theorem}

\begin{remark}[Absorbing extensions]
Let $\mathfrak A$ and $\mathfrak B$ be separable $C^\ast$-algebras with $\mathfrak B$ stable (and $\mathfrak A$ unital). A (unital) extension $\mathfrak e$ of $\mathfrak A$ by $\mathfrak B$ is called \emph{absorbing} if $\mathfrak e \sim_\s \mathfrak e \oplus \mathfrak f$ for any trivial (unital) extension $\mathfrak f$ of $\mathfrak A$ by $\mathfrak B$.\footnote{Just as with triviality, there is a difference between requiring that an extension is absorbing, or that a unital extension is absorbing. Sometimes absorbing unital extensions are said to be unital-absorbing. However, we simply call these absorbing as there is no cause of confusion, since a unital extension can never be absorbing in the general sense (it would have to absorb the extension with zero Busby map, which is impossible).}

By \cite{Thomsen-absorbing} there always exists an absorbing, trivial (unital) extension $\mathfrak e_0$ of $\mathfrak A$ by $\mathfrak B$.\footnote{This requires that $\mathfrak A$ and $\mathfrak B$ are separable. Although the definition of absorption makes sense without separability, we stick to this case.} In particular, $\mathfrak e \oplus \mathfrak e_0$ is absorbing for any (unital) extension $\mathfrak e$. 

In particular, if $\mathfrak e_1$ and $\mathfrak e_2$ are absorbing extensions of $\mathfrak A$ by $\mathfrak B$ with $[\mathfrak e_1] = [\mathfrak e_2]$ in $\Ext(\mathfrak A, \mathfrak B)$, then $\mathfrak e_1 \sim_\s \mathfrak e_2$.

Similarly, if $\mathfrak e_1$ and $\mathfrak e_2$ are absorbing \emph{unital} extensions of $\mathfrak A$ by $\mathfrak B$ with $[\mathfrak e_1]_\s = [\mathfrak e_2]_\s$ in $\Ext_\us(\mathfrak A, \mathfrak B)$ (resp.~$[\mathfrak e_1]_\w = [\mathfrak e_2]_\w$ in $\Ext_\uw(\mathfrak A, \mathfrak B)$), then $\mathfrak e_1 \sim_\s \mathfrak e_2$ (resp.~$\mathfrak e_1 \sim_\w \mathfrak e_2$).
\end{remark}

\begin{remark}[Determining absorption]\label{r:purelylarge}
A priori, it seems inconceivable that one could ever determine when an extension is absorbing. However, this was done by Elliott and Kucerovsky in \cite{ElliottKucerovsky-extensions}.

Following \cite{ElliottKucerovsky-extensions}, an extension $\mathfrak e : 0 \to \mathfrak B \to \mathfrak E \to \mathfrak A \to 0$ of separable $C^\ast$-algebras is called \emph{purely large} if for any $x\in \mathfrak E \setminus \mathfrak B$, there exists a stable $C^\ast$-subalgebra $\mathfrak D \subseteq \overline{x^\ast \mathfrak B x}$ such that $\overline{\mathfrak B \mathfrak D \mathfrak B} = \mathfrak B$. 

By a remarkable result \cite[Theorem 6]{ElliottKucerovsky-extensions}, if $\mathfrak e : 0 \to \mathfrak B \to \mathfrak E \to \mathfrak A \to 0$ is a unital extension of separable $C^\ast$-algebras for which $\mathfrak A$ is nuclear and $\mathfrak B$ is stable, then $\mathfrak e$ is absorbing (in the unital sense) if and only if it is purely large. Similar conditions for when non-unital extensions are absorbing were studied in \cite{Gabe-nonunitalext}.

A separable $C^\ast$-algebra $\mathfrak B$ is said to have the \emph{corona factorisation property} if any full projection $p\in \multialg{\mathfrak B \otimes \mathbb K}$ is equivalent to $1_{\multialg{\mathfrak B\otimes \mathbb K}}$. Many classes of separable $C^\ast$-algebras are known to have the corona factorisation property, e.g.~all $C^\ast$-algebras with finite nuclear dimension by \cite[Corollary 3.5]{Robert-nucdim} (building on the work in \cite{OrtegaPereraRordam-cfp}). In particular, any AF algebra has the corona factorisation property, as these have nuclear dimension zero.

An extension $\mathfrak e$ of $\mathfrak A$ by $\mathfrak B$ with Busby map $\tau \colon \mathfrak A \to \corona{\mathfrak B}$ is called \emph{full} if for every non-zero $a\in \mathfrak A$, $\tau(a)$ generates all of $\corona{\mathfrak B}$ as a two-sided, closed ideal.
As observed by Kucerovsky and Ng in \cite{KucerovskyNg-corona}, if $\mathfrak e : 0 \to \mathfrak B \to \mathfrak E \to \mathfrak A \to 0$ is a full extension of separable $C^\ast$-algebras, for which $\mathfrak B$ is stable and has the corona factorisation property, then $\mathfrak e$ is purely large.
\end{remark}

%%%%%%%%%%%%%%%%%%%%%%%%%%%%%%%%%%%%%%%%%%%%%%%%%%%%%%%%%%%%%%%%%%
%%%%%%%%%%%%%%%%%%%%%%%%%%%%%%%%%%%%%%%%%%%%%%%%%%%%%%%%%%%%%%%%%%

\section{$K$-theory of unital extensions}

The purpose of this section is to collect some results on the $K$-theory of extensions of $C^\ast$-algebras, with a main focus on what happens to the unit in the $K_0$-groups under certain operations of unital extensions. While most results in this section are quite elementary and most likely well-known to some experts in the field, we know of no references to these results and have included detailed proofs for completion.

Consider two six-term exact sequences
\[
\xymatrix{
\mathsf x^{(i)}:  & H_0^{(i)} \ar[r] & L_0^{(i)} \ar[r] & G_0^{(i)} \ar[d] \\
& G_1^{(i)} \ar[u] & L_1^{(i)} \ar[l] & H_1^{(i)} \ar[l]
}
\]
for $i=1,2$. A homomorphism $(\psi_\ast , \rho_\ast , \phi_\ast) \colon \mathsf x^{(1)} \to \mathsf x^{(2)}$ of six-term exact sequences consists of homomorphisms
\[
\phi_\ast \colon G_\ast^{(1)} \to G_\ast^{(2)}, \qquad \psi_\ast \colon H_\ast^{(1)} \to H_\ast^{(2)}, \qquad \rho_\ast \colon L_\ast^{(1)} \to L_\ast^{(2)}
\]
making the obvious diagram commute.

We may also consider six-term exact sequences with certain distinguished elements, which in our case will always be elements in $x_i \in L_0^{(i)}$ and $y_i \in G_0^{(i)}$ for $i=1,2$, and will correspond to the classes of the units in our $K_0$-groups. If this is the case, we only consider homomorphisms such that $\rho_0(x_1) = x_2$ and $\phi_0(y_1) = y_2$.

If $G_\ast^{(1)} = G_\ast^{(2)} =: G_\ast$ and $H_\ast^{(1)} = H_\ast^{(2)} =: H_\ast$ then we say that $\mathsf x^{(1)}$ and $\mathsf x^{(2)}$ are \emph{congruent}, written $\mathsf x^{(1)} \equiv \mathsf x^{(2)}$, if there exists a homomorphism of the form $(\id_{H_\ast}, \rho_\ast, \id_{G_\ast}) \colon \mathsf x^{(1)} \to \mathsf x^{(2)}$.
Note that by the five lemma, this forces $\rho_\ast$ to be an isomorphism, but in general many different $\rho_\ast$ can implement a congruence.

If any of the groups in the six-term exact sequences contain distinguished elements, we require that our homomorphisms preserve these elements. In particular, when considering congruence with $x_i \in L_0^{(i)}$ and $y_i \in G_0^{(i)} = G_0$ being our distinguished elements, we only consider the case $y_1 = y_2$.

\begin{definition}
For an extension $\mathfrak e : 0 \to \mathfrak B \to \mathfrak E \to \mathfrak A \to 0$ of (unital) $C^\ast$-algebras, we let $\Ksixnu(\mathfrak e)$ (resp.~$\Ksix(\mathfrak e)$) denote the six-term exact sequence in $K$-theory (resp.~with distinguished elements $[1_\mathfrak{E}] \in K_0(\mathfrak E)$ and $[1_{\mathfrak A}] \in K_0(\mathfrak A)$). 
\end{definition}

Note that two extensions $\mathfrak e$ and $\mathfrak f$ can \emph{only} have congruent six-term exact sequences, if the two ideals are \emph{equal} and the two quotients are \emph{equal} (isomorphisms are not enough for the definition to make sense). So both extensions \emph{have to} be extensions of $\mathfrak A$ by $\mathfrak B$ for the definition of congruence to make sense.

The following two lemmas are well-known, but we fill in the proofs for completeness.

\begin{lemma}\label{l:sucong}
Let $\mathfrak e_1$ and $\mathfrak e_2$ be unital extensions of $\mathfrak A$ by $\mathfrak B$ which are strongly unitarily equivalent. Then $\Ksix(\mathfrak e_1) \equiv \Ksix(\mathfrak e_2)$.
\end{lemma}
\begin{proof}
If $u\in \multialg{\mathfrak B}$ implements the strong unitary equivalence, then applying $K$-theory to the diagram \eqref{eq:suediag} and using that $K_\ast(\Ad u) = \id_{K_\ast(\mathfrak B)} \colon K_\ast(\mathfrak B) \to K_\ast(\mathfrak B)$, one obtains a congruence $\Ksix(\mathfrak e_1) \equiv \Ksix(\mathfrak e_2)$.
\end{proof}

\begin{lemma}\label{l:0cong}
Let $\mathfrak A$ and $\mathfrak B$ be $C^\ast$-algebras with $\mathfrak A$ unital and $\mathfrak B$ stable. Let $\mathfrak e \colon 0 \to \mathfrak B \to \mathfrak E \to \mathfrak A \to 0$ be a unital extension, and let $\mathfrak e_0$ be a trivial unital extension of $\mathfrak A$ by $\mathfrak B$. Then $\Ksix(\mathfrak e)$ and $\Ksix(\mathfrak e \oplus \mathfrak e_0)$ are congruent.
\end{lemma}
\begin{proof}
Let $s_1,s_2\in \multialg{\mathfrak B}$ be $\mathcal O_2$-isometries so that $\mathfrak e \oplus \mathfrak e_0 = \mathfrak e \oplus_{s_1,s_2} \mathfrak e_0$. Let $\pi \colon \mathfrak E \to \mathfrak A$ be the quotient map, $\sigma \colon \mathfrak E \to \multialg{\mathfrak B}$ be the canonical unital $\ast$-homomorphism, and $\phi \colon \mathfrak A \to \multialg{\mathfrak B}$ be a unital $\ast$-homo\-morphism which lifts $\tau_0$.

The extension algebra $\mathfrak F$ of $\mathfrak e \oplus_{s_1,s_2} \mathfrak e_0$ is by definition
\[
\mathfrak F = \{ a \oplus m \in \mathfrak A \oplus \multialg{\mathfrak B} : \pi_{\mathfrak B}(s_1) \tau(a) \pi_{\mathfrak B}(s_1)^\ast + \pi_{\mathfrak B}(s_2) \tau_0(a) \pi_{\mathfrak B}(s_2)^\ast = \pi_{\mathfrak B}(m) \}.
\]
Define the unital $\ast$-homo\-morphism $\Psi \colon \mathfrak E \to \mathfrak F$ by
\[
\Psi(y) = \pi(y) \oplus (s_1 \sigma(y) s_1^\ast + s_2 \phi(\pi(y)) s_2^\ast).
\]
This is clearly well-defined, and induces a unital $\ast$-homomorphism of extensions by
\[
\xymatrix{
\mathfrak e : & 0 \ar[r] & \mathfrak B \ar[d]_{s_1(-)s_1^\ast} \ar[r] & \mathfrak E \ar[d]_\Psi \ar[r] & \mathfrak A \ar@{=}[d] \ar[r] & 0 \\
\mathfrak e\oplus_{s_1,s_2} \mathfrak e_0 : & 0 \ar[r] & \mathfrak B \ar[r] & \mathfrak F \ar[r] & \mathfrak A \ar[r] & 0.
}
\]
As $(s_1(-)s_1^\ast)_\ast = \id_{K_\ast(\mathfrak B)} \colon K_\ast(\mathfrak B) \to K_\ast(\mathfrak B)$, applying $K$-theory to the above diagram induces a congruence $\Ksix(\mathfrak e) \equiv \Ksix(\mathfrak e\oplus_{s_1,s_2} \mathfrak e_0)$.
\end{proof}

\begin{corollary}\label{c:KsixExt}
Let $\mathfrak A$ and $\mathfrak B$ be separable $C^\ast$-algebras with $\mathfrak A$ unital and $B$ stable. Suppose that $\mathfrak e_1$ and $\mathfrak e_2$ are unital extensions of $\mathfrak A$ by $\mathfrak B$ for which $[\mathfrak e_1]_\s = [\mathfrak e_2]_\s$ in $\Ext_{\us}(\mathfrak A, \mathfrak B)$. Then $\Ksix(\mathfrak e_1) \equiv \Ksix(\mathfrak e_2)$.
\end{corollary}
\begin{proof}
By definition of $\Ext_{\us}$, there are trivial, unital extensions $\mathfrak f_1,\mathfrak f_2$, such that $\mathfrak e_1 \oplus \mathfrak f_1$ and $\mathfrak e_2 \oplus \mathfrak f_2$ are strongly unitarily equivalent. Hence the result follows from Lemmas \ref{l:sucong} and \ref{l:0cong}.
\end{proof}

\begin{lemma}\label{l:moveunit}
  Let $\mathfrak e : 0 \to \mathfrak B \xrightarrow{\iota} \mathfrak E \xrightarrow{\pi} \mathfrak A \to 0$ be a unital extension $C^\ast$-algebras with boundary map $\delta_\ast \colon K_\ast(\mathfrak A) \to K_{1-\ast}(\mathfrak B)$ in $K$-theory, let $u\in \corona{\mathfrak B}$ be a unitary, and let $\chi_1 \colon K_1(\corona{\mathfrak B}) \to K_0(\mathfrak B)$ denote the index map in $K$-theory.
 Then $\Ksix(\Ad u \cdot \mathfrak e)$ (see Definition \ref{d:pbpo}) is congruent to
 \[
  \xymatrix{
   K_0(\mathfrak B) \ar[r]^{\iota_0 \qquad \quad \quad} & (K_0(\mathfrak E), [1_{\mathfrak E}] + \iota_0(\chi_1([u]))) \ar[r]^{\qquad \quad \pi_0} & (K_0(\mathfrak A), [1_\mathfrak{A}]) \ar[d]^{\delta_0} \\
   K_1(\mathfrak A) \ar[u]^{\delta_1} & K_1(\mathfrak E) \ar[l]_{\pi_1} & K_1(\mathfrak B). \ar[l]_{\iota_1}
  }
 \]
\end{lemma}
\begin{proof}

Let $a\in \multialg{\mathfrak B}$ be a lift of $u$ with $\| a\| = 1$, and define
\[
v := \left( \begin{array}{cc} a & 0 \\ (1-a^\ast a)^{1/2} & \; 0 \end{array} \right) \in M_2(\multialg{\mathfrak B}), \qquad v_c := \left( \begin{array}{c} a \\ (1-a^\ast a)^{1/2} \end{array} \right) \in M_{2,1}(\multialg{\mathfrak B}).
\]
Then $v$ is a partial isometry for which $v^\ast v = 1_{\multialg{\mathfrak B}} \oplus 0$. It is well-known, see e.g.~\cite[Section 9.2]{Rordam-book-K-theory}, that 
\begin{equation}\label{eq:chi1u1}
\chi_1([u]) = [1_{M_2(\tilde{\mathfrak B})} - v v^\ast] - [0\oplus 1_{\widetilde{\mathfrak B}}] \in K_0(\mathfrak B).
\end{equation}
Let $\tau$ denote the Busby map of $\mathfrak e$, and identify $\mathfrak E$ with the pull-back $\mathfrak A \oplus_{\tau, \pi_{\mathfrak B}} \multialg{\mathfrak B}$. Define
\[
\mathfrak E_2 := \{ a\oplus y \in \mathfrak A \oplus M_2(\multialg{\mathfrak B}) : (\Ad u \circ \tau (a)) \oplus 0 = M_2(\pi_{\mathfrak B})(y) \in M_2(\corona{\mathfrak B})\},
\]
i.e.~$\mathfrak E_2$ is the pull-back $\mathfrak A \oplus_{(\Ad u \circ \tau)\oplus 0, M_2(\pi_{\mathfrak B})} M_2(\multialg{\mathfrak B})$.
We obtain an embedding
\[
\Ad (1 \oplus v_c) \colon \mathfrak A \oplus_{\tau , \pi_{\mathfrak B}} \multialg{\mathfrak B} \to \mathfrak E_2.
\]
Similarly, identify the extension algebra $\mathfrak E_u$ of $\Ad u \cdot \mathfrak e$ with the pull-back $\mathfrak A \oplus_{\Ad u \circ \tau, \pi_{\mathfrak B}} \multialg{\mathfrak B}$. The embedding $\multialg{\mathfrak B} \to M_2(\multialg{\mathfrak B})$ into the $(1,1)$-corner induces an embedding
\[
\id_{\mathfrak A} \oplus j \colon \mathfrak A \oplus_{\Ad u \circ \tau, \pi_{\mathfrak B}} \multialg{\mathfrak B} \to \mathfrak E_2.
\]
We get the following diagram where all rows are short exact sequences and all maps are $\ast$-homomorphisms
\[
\xymatrix{
0 \ar[r] & \mathfrak B \ar[d]^{\Ad v_c} \ar[r]^\iota & \mathfrak E \ar[d]^{\Ad (1\oplus v_c)} \ar[rr]^{\pi} && \mathfrak A \ar@{=}[d] \ar[r] & 0 \\
0 \ar[r] & M_2(\mathfrak B) \ar[r]^{\quad \iota^{(2)}} & \mathfrak E_2 \ar[rr] && \mathfrak A \ar[r] & 0 \\
0 \ar[r] & \mathfrak B \ar[r] \ar[u]_{j} & \mathfrak E_u \ar[rr] \ar[u]_{\id_{\mathfrak A} \oplus j} && \mathfrak A \ar[r] \ar@{=}[u] & 0.
}
\]
Note that $(\Ad v_c)_\ast, j_\ast \colon K_\ast(\mathfrak B) \to K_\ast(M_2(\mathfrak B))$ are the same map, namely the canonical isomorphism. In particular, by considering the induced maps of six-term exact sequences, the five lemma implies that $(\id_{\mathfrak A} \oplus j)_\ast \colon K_\ast(\mathfrak E_u) \to K_\ast(\mathfrak E_2)$ and $\Ad(1\oplus v_c)_\ast \colon K_\ast(\mathfrak E) \to K_\ast(\mathfrak E_2)$ are isomorphisms. As $(\Ad v_c)_\ast = j_\ast$, it follows that
\[
\Ad (1\oplus v_c)_\ast^{-1} \circ (\id_{\mathfrak A} \oplus j)_\ast \colon K_\ast(\mathfrak E_u) \to K_\ast(\mathfrak E)
\]
induces a congruence $\Ksixnu(\Ad u \cdot \mathfrak e) \equiv \Ksixnu(\mathfrak e)$ which does not necessarily preserve the class of the unit since $\Ad(1\oplus v_c)$ and $\id_\mathfrak{A} \oplus j$ are not unital maps. Thus it remains to prove that 
\[
\Ad (1\oplus v_c)_0^{-1} ((\id_{\mathfrak A} \oplus j)_0([1_{\mathfrak E_u}])) = [1_{\mathfrak E}] + \iota_0(\chi_1([u])),
\]
or alternatively, that
\begin{equation}\label{eq:Ad1oplusvc}
(\Ad (1\oplus v_c))_0 ([1_{\mathfrak E}] + \iota_0(\chi_1([u]))) = (\id_{\mathfrak A} \oplus j)_0([1_{\mathfrak E_u}]) = [1_{\mathfrak A}\oplus (1_{\multialg{\mathfrak B}} \oplus 0)] \in K_0(\mathfrak E_2).
\end{equation}
Note that the unitisation 
\[
\widetilde{\mathfrak E}_2 = \mathfrak E_2 + \mathbb C(0_{\mathfrak A} \oplus (0\oplus 1_{\multialg{\mathfrak B}})) \subseteq \mathfrak A \oplus M_2(\multialg{\mathfrak B}).
\]
As $(\Ad v_c)_0=j_0 \colon K_0(\mathfrak B) \xrightarrow \cong K_0(M_2(\mathfrak B))$ is the canonical isomorphism, it follows from \eqref{eq:chi1u1} (using that $1_{\mathfrak A} \oplus vv^\ast \in \mathfrak E_2$) that
\begin{eqnarray}
\Ad (1\oplus v_c)_0 \circ  \iota_0(\chi_1([u])) &=&  \iota_0^{(2)} \circ j_0 (\chi_1([u])) \nonumber\\
&=& [1_{\widetilde{\mathfrak E}_2} - (1_\mathfrak{A} \oplus v v^\ast)] - [0_{\mathfrak A} \oplus (0\oplus 1_{\multialg{\mathfrak B}})] \nonumber\\
&=& [1_{\mathfrak A} \oplus (1_{\multialg{\mathfrak B}} \oplus 0)] - [1_{\mathfrak A} \oplus vv^\ast] \in K_0(\mathfrak E_2). \label{eq:Ad1oplusvc2}
\end{eqnarray}
Clearly 
\[
\Ad (1\oplus v_c)_0([1_{\mathfrak E}]) = [1_{\mathfrak A} \oplus v_c v_c^\ast] = [1_{\mathfrak A} \oplus v v^\ast] \in K_0(\mathfrak E_2),
\]
and combining this with \eqref{eq:Ad1oplusvc2} yields \eqref{eq:Ad1oplusvc}.
\end{proof}

Recall that if $L_1,L_2$ and $G$ are abelian groups and $\phi_i \colon L_i \to G$ are homomorphisms, then
\[
L_1 \oplus_{\phi_1,\phi_2} L_2 = \{ x_1 \oplus x_2 \in L_1 \oplus L_2 : \phi_1(x_1) = \phi_2(x_2)\}
\]
is the pull-back. When there is no doubt of what the maps $\phi_i$ are, we simply write $L_1 \oplus_G L_2$ instead of $L_1 \oplus_{\phi_1,\phi_2} L_2$.

\begin{remark}\label{r:Baersum}
Recall that if $\mathsf x_i : 0 \to H \xrightarrow{\iota^{(i)}} L_i \xrightarrow{\pi^{(i)}} G \to 0$ are extensions of abelian groups for $i=1,2$, then their \emph{Baer sum} $\mathsf x_1 \oplus \mathsf x_2$ is the extension given by
\[
0 \to H \xrightarrow{\iota^{(1)}} \frac{L_1 \oplus_G L_2}{\{ (\iota^{(1)} (x), - \iota^{(2)}(x)) : x \in H\}} \xrightarrow{\pi^{(1)}} G \to 0.
\]
Addition in the group $\Ext(G,H)$ is given by the Baer sum.
\end{remark}

The following proposition is an explicit formula for computing $\Ksix(\mathfrak e_1 \oplus \mathfrak e_2)$ using a similar construction as the Baer sum, when we know that the boundary maps for one of $\mathfrak e_1$ or $\mathfrak e_2$ vanishes.

\begin{proposition}\label{p:sumext}
Let $\mathfrak e_i : 0 \to \mathfrak B  \xrightarrow{\iota^{(i)}} \mathfrak E_i \xrightarrow{\pi^{(i)}} \mathfrak A \to 0$ be unital extensions of $C^\ast$-algebras for $i=1,2$ such that $\mathfrak B$ is stable. 
Let $\delta_\ast^{(i)} \colon K_\ast(\mathfrak A) \to K_{1-\ast}(\mathfrak B)$ denote the boundary map of $\mathfrak e_i$ in $K$-theory for $i=1,2$. If $\delta_\ast^{(2)} = 0$, then $\Ksix(\mathfrak e_1 \oplus  \mathfrak e_2)$ is congruent to
 \[
  \xymatrix{
  K_0(\mathfrak B) \ar[rr]^{\iota_0^{(1)}\qquad \qquad \qquad \quad} && \left( \frac{K_0(\mathfrak E_1) \oplus_{K_0(\mathfrak A)}K_0(\mathfrak E_2)}{\{ (\iota_0^{(1)}(x),-\iota_0^{(2)}(x)) : x\in K_0(\mathfrak B)\}} , [1_{\mathfrak E_1}] \oplus [1_{\mathfrak E_2}] \right) \ar[rr]^{\qquad \qquad \qquad \pi_0^{(1)}} && (K_0(\mathfrak A) \ar[d]^{\delta_0^{(1)}}, [1_{\mathfrak A}]) \\
  K_1(\mathfrak A) \ar[u]^{\delta_1^{(1)}} && \frac{K_1(\mathfrak E_1) \oplus_{K_1(\mathfrak A)}K_1(\mathfrak E_2)}{\{ (\iota_1^{(1)}(y),- \iota_1^{(2)}(y)) : y\in K_1(\mathfrak B)\}} \ar[ll]_{\pi_1^{(1)} \qquad \quad} && K_1(\mathfrak B). \ar[ll]_{\qquad \qquad \iota_1^{(1)}}
  }
 \]
The same result also holds in the not necessarily unital case by removing all units from the statement.
\end{proposition}
\begin{proof}
For the not necessarily unital case, one simply ignores any mentioning of units in the argument below.

We fix $\mathcal O_2$-isometries $s_1,s_2\in \multialg{\mathfrak B}$, and identify $\mathfrak e_1 \oplus \mathfrak e_2$ with $\mathfrak e_1 \oplus_{s_1,s_2} \mathfrak e_2$, which we denote as $0 \to \mathfrak B \to \mathfrak E \to \mathfrak A \to 0$. Construct the pull-back diagram
\begin{equation}\label{eq:pullback}
 \xymatrix{
  & \mathfrak B \ar@{=}[r] \ar@{>->}[d] & \mathfrak B \ar@{>->}[d]^{\iota^{(1)}} \\
  \mathfrak B \ar@{>->}[r] \ar@{=}[d] & \mathfrak E_0 \ar@{->>}[r] \ar@{->>}[d] & \mathfrak E_1 \ar@{->>}[d]^{\pi^{(1)}} \\
  \mathfrak B \ar@{>->}[r]^{\iota^{(2)}} & \mathfrak E_2 \ar@{->>}[r]^{\pi^{(2)}} & \mathfrak A.
 }
\end{equation}
Applying $K$-theory to this diagram, and using that $\delta_\ast^{(2)} = 0$, one gets the following commutative diagram with exact rows and columns
\[
  \xymatrix{
&& K_0(\mathfrak B) \ar[d] \ar@{=}[r] & K_0(\mathfrak B)\ar[d]^{\iota_0^{(1)}} & \\
  K_1(\mathfrak E_1) \ar[r]^0 \ar[d]^{\pi_1^{(1)}} & K_0(\mathfrak B) \ar[r] \ar@{=}[d] & K_0(\mathfrak E_0) \ar[r] \ar[d] & K_0(\mathfrak E_1) \ar[d]^{\pi_0^{(1)}} \ar[r]^0 & K_1(\mathfrak B) \ar@{=}[d] \\
 K_1(\mathfrak A) \ar[r]^0 & K_0(\mathfrak B) \ar[r]^{\iota_0^{(2)}} & K_0(\mathfrak E_2) \ar[r]^{\pi_0^{(2)}} & K_0(\mathfrak A) \ar[r]^0 & K_1(\mathfrak B).
 }
\]
Hence $K_0(\mathfrak E_0) \cong K_0(\mathfrak E_1) \oplus_{K_0(\mathfrak A)} K_0(\mathfrak E_2)$ canonically, and this isomorphism takes $[1_{\mathfrak E_0}] \in K_0(\mathfrak E_0)$ to the element $[1_{\mathfrak E_1}] \oplus [1_{\mathfrak E_2}] \in K_0(\mathfrak E_1) \oplus_{K_0(\mathfrak A)} K_0(\mathfrak E_2)$.

The pull-back diagram \eqref{eq:pullback} induces a short exact sequence $\mathfrak e_0 : 0 \to \mathfrak B \oplus \mathfrak B \to \mathfrak E_0 \to \mathfrak A \to 0$ where $\mathfrak B \oplus 0$ is the ``top $\mathfrak B$'' and $0\oplus \mathfrak B$ is the ``left $\mathfrak B$'' in \eqref{eq:pullback}.
Let $\Phi \colon \mathfrak B \oplus \mathfrak B \to \mathfrak B$ be the Cuntz sum map $\Phi(b_1 \oplus b_2) = s_1 b_1 s_1^\ast + s_2 b_2 s_2^\ast$. 
We obtain a commutative diagram with exact rows
\begin{equation}\label{eq:pushout}
 \xymatrix{
 0 \ar[r] & \mathfrak B \oplus \mathfrak B \ar[d]^{\Phi} \ar[r] & \mathfrak E_0 \ar[r] \ar[d] & \mathfrak A \ar@{=}[d] \ar[r] &0 \\
 0 \ar[r] & \mathfrak B \ar[r] & \mathfrak E \ar[r] & \mathfrak A \ar[r] & 0,
 }
\end{equation}
for which the $\ast$-homomorphism $\mathfrak E_0 \to \mathfrak E$ is unital.
Applying $K$-theory to this diagram, and using the canonical identification $K_0(\mathfrak E_0) \cong K_0(\mathfrak E_1) \oplus_{K_0(\mathfrak A)} K_0(\mathfrak E_2)$ as well as the fact that $\delta_\ast^{(2)} =0$, one obtains the following commutative diagram with exact rows
\[
 \xymatrix{
 K_1(\mathfrak A) \ar@{=}[d] \ar[r]^{\delta^{(1)}_1 \times 0 \qquad \;} & K_0(\mathfrak B) \oplus K_0(\mathfrak B) \ar[d]^{\mathrm{Sum}} \ar[r]^{(\iota_0^{(1)}, \iota_0^{(2)})\quad \;} & K_0(\mathfrak E_1) \oplus_{K_0(\mathfrak A)} K_0(\mathfrak E_2) \ar[d] \ar[r]^{\qquad \quad \; \; \pi_0^{(1)}} & K_0(\mathfrak A) \ar@{=}[d] \ar[r]^{\quad \; \; \delta^{(1)}_0\times 0 \qquad} 
 & K_1(\mathfrak B)^2 \ar[d]^{\mathrm{Sum}} \\
 K_1(\mathfrak A) \ar[r]^{\quad \delta^{(1)}_1 \quad} & K_0(\mathfrak B) \ar[r] & K_0(\mathfrak E) \ar[r] & K_0(\mathfrak A) \ar[r]^{\quad \delta^{(1)}_0 \quad} & K_1(\mathfrak B).
 }
\]
A diagram chase shows that $K_0(\mathfrak E_1) \oplus_{K_0(\mathfrak A)} K_0(\mathfrak E_2) \to K_0(\mathfrak E)$ is surjective, with kernel  $\{ (\iota_0^{(1)}(x) , - \iota_0^{(2)}(x)) : x\in K_0(\mathfrak B)\}$. As the map $\mathfrak E_0 \to \mathfrak E$ was unital, $[1_{\mathfrak E_1}] \oplus [1_{\mathfrak E_2}]$ is mapped to $[1_{\mathfrak E}]$. Hence we obtain the following commutative diagram with exact rows
\[
 \xymatrix{
 K_1(\mathfrak A) \ar@{=}[d] \ar[r]^{\delta^{(1)}_1 \times 0 \qquad \;} & \frac{K_0(\mathfrak B) \oplus K_0(\mathfrak B)}{\ker \mathrm{Sum}} \ar[d]^{\mathrm{Sum}}_{\cong} \ar[r]^{(\iota_0^{(1)}, \iota_0^{(2)})\qquad} & \frac{K_0(\mathfrak E_1) \oplus_{K_0(\mathfrak A)} K_0(\mathfrak E_2)}{\{ (\iota_0^{(1)}(x) , - \iota_0^{(2)}(x)) : x\in K_0(\mathfrak B)\}} \ar[d]_\cong \ar[r]^{\qquad \quad \; \; \pi_0^{(1)}} & K_0(\mathfrak A) \ar@{=}[d] \ar[r]^{\quad \; \; \delta^{(1)}_0\times 0 \qquad} 
 & \frac{K_1(\mathfrak B)^2}{\ker \mathrm{Sum}} \ar[d]^{\mathrm{Sum}}_\cong \\
 K_1(\mathfrak A) \ar[r]^{\quad \delta^{(1)}_1 \quad} & K_0(\mathfrak B) \ar[r] & K_0(\mathfrak E) \ar[r] & K_0(\mathfrak A) \ar[r]^{\quad \delta^{(1)}_0 \quad} & K_1(\mathfrak B).
 }
\]
The element $[1_{\mathfrak E}]$ exactly corresponds to $[1_{\mathfrak E_1}] \oplus [1_{\mathfrak E_2}]$ via the above isomorphism. By identifying $K_0(\mathfrak B)$ with $\frac{K_0(\mathfrak B) \oplus K_0(\mathfrak B)}{\ker \mathrm{Sum}}$ via the map $x \mapsto (x,0)$, one obtains part of the desired congruence. Running the same argument as above where one interchange $K_0$ and $K_1$, one obtains the rest of the congruence.
\end{proof}

%%%%%%%%%%%%%%%%%%%%%%%%%%%%%%%%%%%%%%%%%%%%%%%%%%%%%%%%%%%%%%%%%%%%
%%%%%%%%%%%%%%%%%%%%%%%%%%%%%%%%%%%%%%%%%%%%%%%%%%%%%%%%%%%%%%%%%%%%
 
\section{A universal coefficient theorem}

Recall that a separable $C^\ast$-algebra $\mathfrak A$ satisfies the UCT (in $\kk$-theory) if and only if there is a short exact sequence
\begin{equation}\label{eq:UCTExt}
0 \to \Ext(K_\ast(\mathfrak A), K_\ast(\mathfrak B)) \to \Ext^{-1}(\mathfrak A, \mathfrak B) \xrightarrow{\gamma_{\mathfrak A, \mathfrak B}} \Hom(K_\ast(\mathfrak A), K_{1-\ast}(\mathfrak B)) \to 0
\end{equation}
for every separable $C^\ast$-algebra $\mathfrak B$. Here we made the canonical identification $\kk^1(\mathfrak A , \mathfrak B) \cong \Ext^{-1}(\mathfrak A, \mathfrak B)$, see Theorem \ref{t:KKExt}. In this section we prove universal coefficient theorems for the unital $\Ext$-groups $\Ext^{-1}_{\us}$ and $\Ext^{-1}_{\uw}$. Such UCT's were stated in \cite{Skandalis-unitalExt} without a proof, and were proved in \cite{Wei-classext} under the assumption that $\mathfrak B$ has an approximate identity of projections.\footnote{While this isn't stated explicitly in \cite[Theorems 4.8 and 4.9]{Wei-classext}, it can be deduced from the proof that $\mathfrak B$ is assumed to have an approximate identity of projections.} We give a complete proof without this additional assumption and prove that the UCT's are natural in both variables. Naturality is crucial for our applications, and was not established in \cite{Wei-classext}.

\begin{definition}
 Given abelian groups $K, H$ and an element $h\in H$, we can form the pointed $\Ext$-group of $(H,h)$ by $K$ by considering pointed extensions
 \[
  0 \to K \to (G,g) \xrightarrow{\phi} (H,h) \to 0
 \]
 for which $\phi(g) = h$. The set $\Ext((H,h),K)$ of congruence classes of such extensions is an abelian group as in the classical case with $\Ext(H,K)$, see Remark \ref{r:Baersum}.
\end{definition}

\begin{remark}\label{r:ptExt}
There is a homomorphism $K \to \Ext((H,h),K)$ given by
\[
k \mapsto [K \rightarrowtail (K\oplus H, k\oplus h) \twoheadrightarrow (H,h)].
\]
The kernel of this map is $\{ \psi(h) : \psi \in \Hom(H,K)\}$. It easily follows that there is a short exact sequence
 \[
0 \to K/\{ \psi(h) : \psi \in \Hom(H,K)\} \to \Ext((H,h), K) \to \Ext(H,K) \to 0.
 \]
\end{remark}

\begin{notation}
For abelian groups $H$ and $K$, and $h\in H$, we let $\Hom((H,h),K)$ denote the subgroup of $\Hom(H,K)$ consisting of homomorphisms $\delta$ for which $\delta(h) = 0$.
\end{notation}
 
\begin{notation}
 We write $\Ext((K_\ast(\mathfrak A),[1_\mathfrak{A}]),K_\ast(\mathfrak B))$ for the group
 \[
  \Ext((K_0(\mathfrak A),[1_\mathfrak{A}]),K_0(\mathfrak B)) \oplus \Ext(K_1(\mathfrak A),K_1(\mathfrak B))
 \]
and $\Hom((K_\ast(\mathfrak A),[1_\mathfrak{A}]),K_{\ast +1} (\mathfrak B))$ for the group
 \[
  \Hom((K_0(\mathfrak A),[1_\mathfrak{A}]),K_1(\mathfrak B)) \oplus \Hom(K_1(\mathfrak A), K_0(\mathfrak B)).
  \]
\end{notation}
 
\begin{remark}\label{r:gammakappa}
 It is easily seen that there is a homomorphism 
\[
\tilde \gamma_{\mathfrak A, \mathfrak B} \colon \Ext_{\us}^{-1}(\mathfrak A, \mathfrak B) \to \Hom((K_\ast(\mathfrak A),[1_\mathfrak{A}]),K_{\ast +1} (\mathfrak B)),
\]
given by mapping $[\mathfrak e]_\s$ to its boundary map in $K$-theory.

Similarly, there is a map
\[
\widetilde \kappa_{\mathfrak A, \mathfrak B} \colon \ker \widetilde \gamma_{\mathfrak A, \mathfrak B} \to \Ext((K_\ast(\mathfrak A), [1_{\mathfrak A}]), K_\ast(\mathfrak B))
\]
given by mapping $[\mathfrak e]_\s$ to its induced six-term exact sequence in $K$-theory with position of the unit. This is well defined since the boundary maps vanish, but a priori it is not obviously a homomorphism (it is a homomorphism by Corollary \ref{c:kappatilde} below).
\end{remark}

The following is an immediate consequence of Proposition \ref{p:sumext} and the definition of the sum in the pointed $\Ext$-group.

\begin{corollary}\label{c:kappatilde}   
Let $\mathfrak A$ and $\mathfrak B$ be separable $C^\ast$-algebras for which $\mathfrak A$ is unital. Then the map
\[
\widetilde \kappa_{\mathfrak A, \mathfrak B} \colon \ker \widetilde \gamma_{\mathfrak A, \mathfrak B} \to \Ext((K_\ast(\mathfrak A),[1_{\mathfrak A}]), K_\ast(\mathfrak B))
\]
defined in Remark \ref{r:gammakappa} is a homomorphism.
\end{corollary}
 
We introduce the following non-standard notation to ease what follows.

 \begin{notation}\label{n:Gamma}
  Let $\mathfrak A$ be a unital separable $C^\ast$-algebra, and $\mathfrak B$ be a separable $C^\ast$-algebra. We define
  \[
   \Gamma_{\mathfrak A, \mathfrak B} := \{ \psi([1_\mathfrak{A}]) : \psi\in \Hom(K_0(\mathfrak A), K_0(\mathfrak B))\}.
  \]
\end{notation}

\begin{remark}
If $\mathfrak A$ and $\mathfrak B$ are $C^\ast$-algebras with $\mathfrak A$ unital, then
\[
0 \to K_0(\mathfrak B) / \Gamma_{\mathfrak A, \mathfrak B} \to \Ext((K_\ast(\mathfrak A),[1_\mathfrak{A}]),K_\ast(\mathfrak B)) \to \Ext(K_\ast(\mathfrak A),K_\ast(\mathfrak B)) \to 0
\]
is a short exact sequence by Remark \ref{r:ptExt}.
\end{remark}

For a unital $C^\ast$-algebra $\mathfrak D$, we let $\mathcal U(\mathfrak D)$ denote its unitary group, and let $\mathcal U_0(\mathfrak D)$ denote the connected component of $1_{\mathfrak D}$ in $\mathcal U(\mathfrak D)$. Recall that a unital $C^\ast$-algebra $\mathfrak D$ is \emph{$K_1$-surjective} (resp.~\emph{$K_1$-injective}) if the canonical homomorphism $\mathcal U(\mathfrak D)/\mathcal U_0(\mathfrak D) \to K_1(\mathfrak D)$ is surjective (resp.~injective), and \emph{$K_1$-bijective} if it is both $K_1$-surjective and $K_1$-injective.

While the following result is well-known to experts, we know of no reference and thus include a proof.

\begin{proposition}\label{p:K1corona}
If $\mathfrak B$ is a stable $C^\ast$-algebra then the corona algebra $\corona{\mathfrak B}$ is $K_1$-bijective.
\end{proposition}
\begin{proof}
Stability of $\mathfrak B$ implies that $\corona{\mathfrak B}$ is properly infinite and thus $K_1$-surjective by \cite{Cuntz-K-theoryI}. For $K_1$-injectivity, let $u\in \mathcal U(\corona{\mathfrak B})$ be such that $[u]= 0$ in $K_1(\corona{\mathfrak B})$. By \cite[Corollary 2.5]{Nistor-stablerank} the connected stable rank of $\mathfrak B$ is at most 2. Consequently the general stable rank\footnote{Not to be confused with the topological stable rank, which in modern terms is usually just referred to as stable rank.} of $\mathfrak B$ is at most 2. By \cite[Theorem 2]{Nagy-liftinginv} (which relies on results in \cite{Rieffel-sr}) it follows that $u$ lifts to $\widetilde u\in \mathcal U(\multialg{\mathfrak B})$. By \cite{CuntzHigson-Kuipersthm} one has $\mathcal U(\multialg{\mathfrak B}) = \mathcal U_0(\multialg{\mathfrak B})$, and thus $u\in \mathcal U_0(\corona{\mathfrak B})$. Hence $\corona{\mathfrak B}$ is $K_1$-injective.
\end{proof}

\begin{remark}\label{r:ex}
Let $\mathfrak A$ and $\mathfrak B$ be separable $C^\ast$-algebras for which $\mathfrak A$ is unital and $\mathfrak B$ is stable. For every $x\in K_0(\mathfrak B) \cong K_1(\corona{\mathfrak B})$ there is an induced semisplit, unital extension $\mathfrak e_x$ of $\mathfrak A$ by $\mathfrak B$ (uniquely determined up to strong unitary equivalence) given as follows: Let $\tau_0 \colon \mathfrak A \to \corona{\mathfrak B}$ be the Busby map of a trivial, absorbing unital extension \cite{Thomsen-absorbing}, and let $u\in \mathcal U (\corona{\mathfrak B})$ be a unitary being mapped to $x$ under the natural isomorphism $K_1(\corona{\mathfrak B}) \xrightarrow \cong K_0(\mathfrak B)$. Then $\mathfrak e_x$ is the extension with Busby map $\Ad u \circ \tau_0$.

As $\tau_0$ is uniquely determined up to strong unitary equivalence, and since $K_1(\corona{\mathfrak B}) = \mathcal U(\corona{\mathfrak B}) / \mathcal U_0 (\corona{\mathfrak B})$ by Proposition \ref{p:K1corona}, it easily follows that $\mathfrak e_x$ is unique up to strong unitary equivalence.
\end{remark}

The following elementary lemma will be used frequently.

\begin{lemma}\label{l:Advssum}
Let $\mathfrak A$ and $\mathfrak B$ be separable $C^\ast$-algebras for which $\mathfrak A$ is unital and $\mathfrak B$ is stable. Let $\mathfrak e$ be a unital extension of $\mathfrak A$ by $\mathfrak B$, and let $u \in \mathcal U (\corona{\mathfrak B})$. Then
\[
[\Ad u \cdot \mathfrak e]_\s = [\mathfrak e]_\s + [\mathfrak e_{[u]_1}]_\s \in \Ext_\us(A, B).
\]
In particular, the map
\[
K_0(\mathfrak B) \to \Ext_\us^{-1} (\mathfrak A, \mathfrak B) , \qquad x\mapsto [\mathfrak e_x]_\s
\]
is a group homomorphism.
\end{lemma}
\begin{proof}
Let $s_1,s_2\in \multialg{\mathfrak B}$ be $\mathcal O_2$-isometries, and let $\oplus$ denote the Cuntz sum induced by this choice of isometries. Then
\begin{equation}
 \Ad (u\oplus u^\ast) \circ (\tau_\mathfrak{e} \oplus \tau_{\mathfrak e_{[u]_1}}) = \Ad (u \oplus u^\ast) \circ (\tau_{\mathfrak e} \oplus (\Ad u \circ \tau_0)) = (\Ad u \circ \tau_{\mathfrak e}) \oplus \tau_0
 \end{equation}
where $\tau_0$ is an absorbing, trivial unital extension. As $u \oplus u^\ast$ lifts to a unitary in $\multialg{\mathfrak B}$, the result follows.
\end{proof}

The following is an immediate consequence of Lemma \ref{l:moveunit} applied to the case where $\mathfrak e$ is a trivial unital extension.

\begin{corollary}\label{c:Extex}
Let $\mathfrak A$ and $\mathfrak B$ be separable $C^\ast$-algebras for which $\mathfrak A$ is unital and $\mathfrak B$ is stable, and let $x\in K_0(\mathfrak B)$. Then $\mathfrak e_x$ induces the element
\[
[0 \to K_0(\mathfrak B) \to (K_0(\mathfrak B) \oplus K_0(\mathfrak A), x \oplus [1_\mathfrak{A}]) \to (K_0(\mathfrak A), [1_{\mathfrak A}]) \to 0]
\]
in $\Ext((K_0(\mathfrak A),[1_{\mathfrak A}]), K_0(\mathfrak B))$.
\end{corollary}

Recall that $\gamma_{\mathfrak A, \mathfrak B} \colon \Ext^{-1}(\mathfrak A, \mathfrak B) \to \Hom(K_\ast(\mathfrak A), K_{1-\ast}(\mathfrak B))$ denotes the canonical homomorphism.

\begin{lemma}\label{l:UCT}
Let $\mathfrak A$ be a separable, unital $C^\ast$-algebra satisfying the UCT, and let $\mathfrak B$ be a separable, stable $C^\ast$-algebra. 
  Then there is an exact sequence
  \[
    0 \to K_0(\mathfrak B)/\Gamma_{\mathfrak A, \mathfrak B} \to \Ext_{\us}^{-1}(\mathfrak A, \mathfrak B) \to \Ext^{-1}(\mathfrak A, \mathfrak B).
  \]
  Moreover, the map $\Ext_{\uw}^{-1}(\mathfrak A, \mathfrak B) \to \Ext^{-1}(\mathfrak A, \mathfrak B)$ is an isomorphism onto
\[
\gamma_{\mathfrak A, \mathfrak B}^{-1}(\Hom((K_\ast(\mathfrak A),[1_\mathfrak{A}]),K_{\ast +1} (\mathfrak B))) \quad (\subseteq \Ext^{-1}(\mathfrak A, \mathfrak B)).
\]
\end{lemma}
\begin{proof}
By a result of Skandalis \cite[Remarque 2.8]{Skandalis-KKnuc} (see also \cite{Skandalis-unitalExt} or \cite{ManuilovThomsen-unitalext} for a proof), there is an exact sequence of the form
 \[
  \xymatrix{
  K_0(\mathfrak B) \ar[r] & \Ext_{\us}^{-1}(\mathfrak A, \mathfrak B) \ar[r] & \Ext^{-1}(\mathfrak A, \mathfrak B) \ar[d]^{\iota_1^\ast} \\
  \kk(\mathfrak A, \mathfrak B) \ar[u]^{\iota_0^\ast} & & K_1(\mathfrak B)
  }
 \]
 where $\iota_i^\ast$ is induced from the unital $\ast$-homomorphism $\iota \colon \mathbb C \to \mathfrak A$. It is easily seen that $\iota_0^\ast \colon \kk(\mathfrak A, \mathfrak B) \to K_0(\mathfrak B)$ factors as
 \[
  \kk(\mathfrak A, \mathfrak B) \xrightarrow{\gamma_0} \Hom(K_0(\mathfrak A), K_0(\mathfrak B)) \xrightarrow{\ev_{[1_\mathfrak A]}} K_0(\mathfrak B)
 \]
 where $\ev_{[1_\mathfrak{A}]}$ is evaluation at $[1_\mathfrak{A}]$. Similarly, $\iota_1^\ast \colon \Ext^{-1}(\mathfrak A, \mathfrak B) \to K_1(\mathfrak B)$ factors as
 \[
  \Ext^{-1}(\mathfrak A, \mathfrak B) \xrightarrow{\gamma_0} \Hom(K_0(\mathfrak A), K_1(\mathfrak B)) \xrightarrow{\ev_{[1_\mathfrak A]}} K_1(\mathfrak B).
 \]
Since $\mathfrak A$ satisfies the UCT, $\gamma_0$ is surjective and thus $\mathrm{im}(\iota_0^\ast) = \Gamma_{\mathfrak A, \mathfrak B}$. 
Hence the exact sequence collapses to an exact sequence
\[
 0 \to K_0(\mathfrak B)/\Gamma_{\mathfrak A, \mathfrak B} \to \Ext_{\us}^{-1}(\mathfrak A, \mathfrak B) \to \Ext^{-1}(\mathfrak A, \mathfrak B)
\]
where the image of $\Ext_{\us}^{-1}(\mathfrak A, \mathfrak B) \to \Ext^{-1}(\mathfrak A, \mathfrak B)$ is $\ker \iota_1^{\ast}$. 
By the above, it easily follows that $\ker \iota_1^{\ast} = \gamma_{\mathfrak A, \mathfrak B}^{-1}(\Hom((K_\ast(\mathfrak A),[1_\mathfrak{A}]),K_{\ast +1} (\mathfrak B)))$, so we obtain a short exact sequence
\[
0 \to K_0(\mathfrak B)/\Gamma_{\mathfrak A, \mathfrak B} \to \Ext_{\us}^{-1}(\mathfrak A, \mathfrak B) \to \gamma_{\mathfrak A, \mathfrak B}^{-1}(\Hom((K_\ast(\mathfrak A),[1_\mathfrak{A}]),K_{\ast +1} (\mathfrak B))) \to 0.
\]

Using Lemma \ref{l:Advssum} it follows that the quotient $\Ext^{-1}_{\us}(\mathfrak A, \mathfrak B)/(K_0(\mathfrak B)/\Gamma_{\mathfrak A, \mathfrak B})$ is canonically isomorphic to $\Ext^{-1}_{\uw}(\mathfrak A, \mathfrak B)$. Combined with the above short exact sequence it follows that $\Ext^{-1}_\uw (\mathfrak A, \mathfrak B) \to \Ext^{-1}(\mathfrak A, \mathfrak B)$ is injective and its image is
\[
\gamma_{\mathfrak A, \mathfrak B}^{-1}(\Hom((K_\ast(\mathfrak A),[1_\mathfrak{A}]),K_{\ast +1} (\mathfrak B)))
\]
as desired.
\end{proof}

We can now assemble the pieces provided by the previous results in this section and obtain the following universal coefficient theorem.  This is a minor improvement on the UCT sequences proved by Wei \cite[Theorems 4.8 and 4.9]{Wei-classext}, in which the $C^\ast$-algebra $\mathfrak B$ was required to have an approximate identity of projections. Also, Wei does not prove that the UCT's for the unital $\Ext$-groups are natural, which will be important in our applications.
  
\begin{theorem}\label{t:sesunitalext}
  Let $\mathfrak A$ be a unital, separable $C^\ast$-algebra satisfying the UCT, and let $\mathfrak B$ be a separable $C^\ast$-algebra. 
 There is a commutative diagram
  \begin{equation}\label{eq:UCTdiagram}
  \xymatrix{
  K_0(\mathfrak B)/\Gamma_{\mathfrak A, \mathfrak B} \ar@{=}[r] \ar@{>->}[d] & K_0(\mathfrak B)/\Gamma_{\mathfrak A, \mathfrak B} \ar@{>->}[d] & \\
  \Ext((K_\ast(\mathfrak A),[1_\mathfrak{A}]),K_\ast(\mathfrak B)) \ar@{->>}[d] \ar@{>->}[r] & \Ext_{\us}^{-1}(\mathfrak A, \mathfrak B) \ar@{->>}[d] \ar@{->>}[r]^{\widetilde \gamma_{\mathfrak A, \mathfrak B} \qquad \qquad} & 
  \Hom((K_\ast(\mathfrak A),[1_\mathfrak{A}]),K_{\ast +1} (\mathfrak B)) \ar@{=}[d] \\
  \Ext(K_\ast(\mathfrak A), K_\ast(\mathfrak B)) \ar@{>->}[r] & \Ext^{-1}_{\uw}(\mathfrak A, \mathfrak B) \ar@{->>}[r]^{\gamma_{\mathfrak A, \mathfrak B} \qquad \qquad} & \Hom((K_\ast(\mathfrak A),[1_\mathfrak{A}]),K_{\ast +1} (\mathfrak B)).
  }
  \end{equation}
  for which all rows and columns are short exact sequences. This diagram is natural with respect to unital $\ast$-homomorphisms in the first vairable, and with respect to non-degenerate $\ast$-homomorphisms in the second variable.
\end{theorem}
\begin{proof}
 By replacing $\mathfrak B$ with $\mathfrak B \otimes \mathbb K$, we may assume that $\mathfrak B$ is stable.

By Lemma \ref{l:UCT} and the UCT for $\Ext^{-1}$ (see \eqref{eq:UCTExt}), we obtain a short exact sequence
\begin{equation}\label{eq:UCTweak}
 0 \to \Ext(K_\ast(\mathfrak A), K_\ast(\mathfrak B)) \to \Ext^{-1}_{\uw}(\mathfrak A, \mathfrak B) \xrightarrow{\gamma_{\mathfrak A, \mathfrak B}} \Hom((K_\ast(\mathfrak A),[1_\mathfrak{A}]),K_{\ast +1} (\mathfrak B)) \to 0.
\end{equation}
The map $\Ext(K_\ast(\mathfrak A), K_\ast(\mathfrak B)) \to \ker \gamma_{\mathfrak A, \mathfrak B}$ above, which is an isomorphism by exactness, is exactly the inverse of the isomorphism
\[
\kappa_{\mathfrak A, \mathfrak B} \colon \ker \gamma_{\mathfrak A, \mathfrak B} \xrightarrow \cong \Ext(K_\ast(\mathfrak A), K_\ast(\mathfrak B))
\]
given by applying $K$-theory to a given extension (which induce short exact sequences by vanishing of the boundary maps). That $\kappa_{\mathfrak A, \mathfrak B}$ is an isomorphism follows from the UCT.
The homomorphism
\[
\widetilde \gamma_{\mathfrak A, \mathfrak B} \colon \Ext^{-1}_\us(\mathfrak A, \mathfrak B) \to \Hom((K_\ast(\mathfrak A),[1_\mathfrak{A}]),K_{\ast +1} (\mathfrak B))
\]
is the composition of the surjective homomorphisms $\Ext^{-1}_\us \to \Ext^{-1}_\uw$ and $\gamma_{\mathfrak A, \mathfrak B}$ from \eqref{eq:UCTweak}, so $\widetilde \gamma_{\mathfrak A, \mathfrak B}$ is surjective. Hence we obtain the commutative diagram
\begin{equation}\label{eq:almostUCT}
\xymatrix{
K_0(\mathfrak B)/\Gamma_{\mathfrak A, \mathfrak B} \ar@{=}[r] \ar@{>->}[d] & K_0(\mathfrak B) / \Gamma_{\mathfrak A, \mathfrak B} \ar@{>->}[d] & \\
\ker \widetilde \gamma_{\mathfrak A, \mathfrak B} \ar@{->>}[d] \ar@{>->}[r] & \Ext^{-1}_\us(\mathfrak A, \mathfrak B) \ar@{->>}[d] \ar@{->>}[r]^{\widetilde \gamma_{\mathfrak A, \mathfrak B} \qquad \qquad} & \Hom((K_\ast(\mathfrak A), [1_{\mathfrak A}]), K_{1-\ast}(\mathfrak B)) \ar@{=}[d] \\
\ker \gamma_{\mathfrak A, \mathfrak B} \ar@{>->}[r] & \Ext^{-1}_\uw (\mathfrak A, \mathfrak B) \ar@{->>}[r]^{\gamma_{\mathfrak A, \mathfrak B}\qquad \qquad} & \Hom((K_\ast(\mathfrak A), [1_{\mathfrak A}]), K_{1-\ast}(\mathfrak B))
}
\end{equation}
for which the rows and columns are short exact sequences. Consider the diagram
\begin{equation}\label{eq:kergammastuff}
\xymatrix{
0 \ar[r] & K_0(\mathfrak B)/ \Gamma_{\mathfrak A, \mathfrak B} \ar@{=}[d] \ar[r] & \ker \widetilde \gamma_{\mathfrak A, \mathfrak B} \ar[d]^{\widetilde \kappa_{\mathfrak A, \mathfrak B}}_{(\cong)} \ar[r] & \ker \gamma_{\mathfrak A, \mathfrak B} \ar[d]^{\kappa_{\mathfrak A, \mathfrak B}}_\cong \ar[r] & 0 \\
0 \ar[r] & K_0(\mathfrak B)/ \Gamma_{\mathfrak A, \mathfrak B} \ar[r] & \Ext((K_\ast(\mathfrak A), [1_{\mathfrak A}]), K_\ast(\mathfrak B)) \ar[r] & \Ext(K_\ast(\mathfrak A), K_\ast(\mathfrak B)) \ar[r] & 0
}
\end{equation}
which has exact rows. The map $\widetilde \kappa_{\mathfrak A, \mathfrak B}$ is a homomorphism by Corollary \ref{c:kappatilde}, and clearly the right square above commutes. The left square above commutes by Remark \ref{r:ptExt} and Corollary \ref{c:Extex}. Hence $\widetilde \kappa_{\mathfrak A, \mathfrak B}$ is an isomorphism by the five lemma. By gluing together the diagrams \eqref{eq:almostUCT} and \eqref{eq:kergammastuff} in the obvious way, we obtain the desired diagram \eqref{eq:UCTdiagram}.

It remains to be shown that the diagram \eqref{eq:UCTdiagram} is natural in both variables. For verifying this let $\mathfrak C$ be separable, unital $C^\ast$-algebra satisfying the UCT, let $\phi \colon \mathfrak C \to \mathfrak A$ be a unital $\ast$-homomorphism, let $\mathfrak D$ be a separable, stable $C^\ast$-algebra, and let $\psi \colon \mathfrak B \to \mathfrak D$ be a non-degenerate $\ast$-homomorphism. We first check that the diagram \eqref{eq:almostUCT} is natural, and then \eqref{eq:kergammastuff}.

It is well-known that $\Ext^{-1}_\us(\mathfrak A, \mathfrak B) \to \Ext^{-1}_\uw(\mathfrak A , \mathfrak B)$ is natural, and by naturality of six-term exact sequences the maps $\widetilde \gamma_{\mathfrak A, \mathfrak B}$ and $ \gamma_{\mathfrak A, \mathfrak B}$ are natural.

Again by naturality of six-term exact sequences, it follows that 
\[
\phi^\ast (\ker \widetilde \gamma_{\mathfrak A, \mathfrak B}) \subseteq \ker \widetilde \gamma_{\mathfrak C, \mathfrak B}, \qquad \textrm{and} \qquad \psi_\ast(\ker \widetilde \gamma_{\mathfrak A, \mathfrak B}) \subseteq \ker  \widetilde \gamma_{\mathfrak A, \mathfrak D}.
\]
Hence the inclusion $\ker \widetilde \gamma_{\mathfrak A, \mathfrak B} \hookrightarrow \Ext^{-1}_\us(\mathfrak A, \mathfrak B)$ is natural in both variables. Similarly, the inclusion $\ker \gamma_{\mathfrak A, \mathfrak B} \hookrightarrow \Ext^{-1}_\uw(\mathfrak A, \mathfrak B)$ and the map $\ker \widetilde \gamma_{\mathfrak A, \mathfrak B} \to \ker \gamma_{\mathfrak A, \mathfrak B}$ are natural in both variables. This implies that the diagram \eqref{eq:almostUCT} is natural. Hence it remains to check that the diagram \eqref{eq:kergammastuff} is natural.

It is straightforward to verify that the maps in the lower row of \eqref{eq:kergammastuff} are natural (this is purely algebraic, and of course uses that $\phi_0([1_{\mathfrak C}]) = [1_{\mathfrak A}]$). 
We saw above that $\ker \widetilde \gamma_{\mathfrak A, \mathfrak B} \to \ker \gamma_{\mathfrak A, \mathfrak B}$ is natural.

We will show that $\widetilde \kappa_{\mathfrak A, \mathfrak B}$ is natural in the first variable. Let $\mathfrak e : 0 \to \mathfrak B \to \mathfrak E \to \mathfrak A \to 0$ be a unital extension inducing an element in $\ker \widetilde \gamma_{\mathfrak A, \mathfrak B}$, i.e.~$\mathfrak e$ has vanishing boundary maps in $K$-theory. Construct the pull-back diagram
\begin{equation}\label{eq:pbnat}
\xymatrix{
\mathfrak e \cdot \phi : & 0 \ar[r] & \mathfrak B \ar[r] \ar@{=}[d] & \mathfrak E_\phi \ar[d] \ar[r] & \mathfrak C \ar[d]^\phi \ar[r] & 0 \\
\mathfrak e : & 0 \ar[r] & \mathfrak B \ar[r] & \mathfrak E \ar[r] & \mathfrak A \ar[r] & 0.
}
\end{equation}
As $\phi$ is a unital map, $\mathfrak E_\phi$ is unital and the map $\mathfrak E_\phi \to \mathfrak E$ is unital. As $\phi^\ast([\mathfrak e]_\s) = [\mathfrak e \cdot \phi]_\s$, we should check that
\begin{equation}\label{eq:kappapb}
\widetilde \kappa_{\mathfrak C, \mathfrak B}([\mathfrak e \cdot \phi]_\s) = (\phi_\ast)^\ast (\widetilde \kappa_{\mathfrak A, \mathfrak B}([\mathfrak e]_\s)).
\end{equation}
Applying $K$-theory to the pull-back diagram \eqref{eq:pbnat}, and using that both $\mathfrak e$ and $\mathfrak e \cdot \phi$  have vanishing boundary maps, we obtain the diagram
\[
\xymatrix{
\widetilde \kappa_{\mathfrak C, \mathfrak B}([\mathfrak e \cdot \phi]) : & 0 \ar[r] & K_\ast(\mathfrak B) \ar[r] \ar@{=}[d] & (K_\ast(\mathfrak E_\phi),[1_{\mathfrak E_\phi}]) \ar[d] \ar[r] & (K_\ast(\mathfrak C), [1_\mathfrak{C}]) \ar[d]^{\phi_\ast} \ar[r] & 0 \\
\widetilde \kappa_{\mathfrak A, \mathfrak B}([\mathfrak e]) : & 0 \ar[r] & K_\ast(\mathfrak B) \ar[r] & (K_\ast(\mathfrak E), [1_{\mathfrak E}]) \ar[r] & (K_\ast(\mathfrak A), [1_{\mathfrak A}]) \ar[r] & 0.
}
\]
Since this is a pull-back diagram it follows that \eqref{eq:kappapb} holds. Hence $\widetilde \kappa_{\mathfrak A, \mathfrak B}$ is natural in the first variable. That $\widetilde \kappa_{\mathfrak A, \mathfrak B}$ is natural in the second variable, and that $\kappa_{\mathfrak A, \mathfrak B}$ is natural in both variables, is checked in a similar fashion.

It remains to check that $K_0(\mathfrak B)/\Gamma_{\mathfrak A, \mathfrak B} \to \ker \widetilde \gamma_{\mathfrak A, \mathfrak B}$ is natural in both variables. For this, fix a unitary in $u\in \corona{\mathfrak B}$ inducing an arbitrary element in $K_0(\mathfrak B)$. Let $\mathfrak{e}_{\mathfrak A, \mathfrak B}$ and $\mathfrak e_{\mathfrak C, \mathfrak B}$ be absorbing, unital extensions of $\mathfrak A$ by $\mathfrak B$ and of $\mathfrak C$ by $\mathfrak B$ respectively. By definition, we have
\[
[u] +  \Gamma_{\mathfrak A, \mathfrak B} \mapsto [\Ad u \cdot \mathfrak e_{\mathfrak A, \mathfrak B}]_\s \in \ker \widetilde \gamma_{\mathfrak A, \mathfrak B}, \qquad [u] + \Gamma_{\mathfrak C, \mathfrak B} \mapsto [\Ad u \cdot \mathfrak e_{\mathfrak C, \mathfrak B}]_\s \in \ker \widetilde \gamma_{\mathfrak C, \mathfrak B}.
\]
In order to check that $K_0(\mathfrak B)/\Gamma_{\mathfrak A, \mathfrak B} \to \ker \widetilde \gamma_{\mathfrak A, \mathfrak B}$ is natural in the first variable, we should therefore verify that
\[
 \phi^\ast([\Ad u \cdot \mathfrak e_{\mathfrak A, \mathfrak B}]_\s) = [\Ad u \cdot \mathfrak e_{\mathfrak C, \mathfrak B}]_\s.
\]
This follows easily from Lemma \ref{l:Advssum} since
\[
 \phi^\ast([\Ad u \cdot \mathfrak e_{\mathfrak A, \mathfrak B}]_\s) = [\Ad u \cdot (\mathfrak e_{\mathfrak A, \mathfrak B} \cdot \phi)]_\s = [\mathfrak e_{\mathfrak A, \mathfrak B} \cdot \phi]_\s + [\Ad u \cdot \mathfrak e_{\mathfrak C, \mathfrak B}]_\s = [\Ad u \cdot \mathfrak e_{\mathfrak C, \mathfrak B}]_\s,
\]
where we used that $\mathfrak e_{\mathfrak A, \mathfrak B} \cdot \phi$ is trivial so that $[\mathfrak e_{\mathfrak A, \mathfrak B} \cdot \phi]_\s = 0$. Hence $K_0(\mathfrak B)/\Gamma_{\mathfrak A, \mathfrak B} \to \ker \widetilde \gamma_{\mathfrak A, \mathfrak B}$ is natural in the first variable. For the second variable, let $\overline \psi \colon \corona{\mathfrak B} \to \corona{\mathfrak D}$ be the induced $\ast$-homomorphism, and let $\mathfrak e_{\mathfrak A, \mathfrak D}$ be an absorbing, unital extension of $\mathfrak A$ by $\mathfrak D$. Note that $\psi_\ast([u]) = [\overline \psi(u)]$. As above, we get
\begin{eqnarray*}
\psi_\ast([\Ad u \cdot \mathfrak e_{\mathfrak A, \mathfrak B}]_\s) &=& [(\overline \psi \circ \Ad u) \cdot \mathfrak e_{\mathfrak A, \mathfrak B}]_\s \\
&=& [\Ad \overline \psi(u) \cdot ( \overline \psi \cdot \mathfrak e_{\mathfrak A, \mathfrak B})]_\s \\
&\stackrel{\textrm{Lem.~}\ref{l:Advssum}}{=}& [\overline \psi \cdot \mathfrak e_{\mathfrak A, \mathfrak B}]_\s + [\Ad \overline \psi(u) \cdot \mathfrak e_{\mathfrak A, \mathfrak D}]_\s \\
&=& [\Ad \overline \psi(u) \cdot \mathfrak e_{\mathfrak A, \mathfrak D}]_\s.
\end{eqnarray*}
As $[\Ad \overline \psi(u) \cdot \mathfrak e_{\mathfrak A, \mathfrak D}]_\s$ is the image of $\psi_\ast([u]) + \Gamma_{\mathfrak A, \mathfrak D}$ via the map $K_0(\mathfrak B)/\Gamma_{\mathfrak A, \mathfrak B} \to \ker \widetilde \gamma_{\mathfrak A, \mathfrak B}$, it follows that this map is natural in the second variable, thus finishing the proof.
 \end{proof}

%%%%%%%%%%%%%%%%%%%%%%%%%%%%%%%%%%%%%%%%%
%%%%%%%%%%%%%%%%%%%%%%%%%%%%%%%%%%%%%%%%%

\section{Classification of unital extensions}

In this section we will apply our universal coefficient theorem to obtain classification results for certain unital extensions of $C^\ast$-algebras via their six-term exact sequence in $K$-theory.

The main idea is the following: suppose $\mathfrak e_1$ and $\mathfrak e_2$ are absorbing, semisplit unital extensions of $\mathfrak A$ by $\mathfrak B$, and suppose that $[\mathfrak e_1]_\w = [\mathfrak e_2]_\w \in \Ext_\uw^{-1}(\mathfrak A, \mathfrak B)$. By Theorem \ref{t:sesunitalext} there is an element $x\in K_0(\mathfrak B)$ such that $[\mathfrak e_1]_\s = [\mathfrak e_2 \oplus \mathfrak e_x]_\s$, and in particular $\mathfrak e_1 \cong \mathfrak e_2 \oplus \mathfrak e_x$ by absorption. So the goal will be to prove, under certain conditions, that $\mathfrak e_2 \oplus \mathfrak e_x \cong \mathfrak e_2$.

As a technical devise, we introduce the following notation.

\begin{notation}\label{n:Gammadelta}
If $\delta_\ast \in \Hom((K_\ast(\mathfrak A),[1_\mathfrak{A}]),K_{\ast +1} (\mathfrak B))$, then we define
  \[
   \Gamma_{\mathfrak A, \mathfrak B}^{\delta_\ast} := q_{\delta_1}^{-1}(\{ \phi([1_\mathfrak{A}]_0) : \phi\in \Hom(\ker \delta_0, \coker \delta_1)\})
  \]
  where $q_{\delta_1} \colon K_0(\mathfrak B) \to \coker \delta_1$ is the canonical epimorphism.
\end{notation}

Note that we always have $\Gamma_{\mathfrak A, \mathfrak B} = \Gamma_{\mathfrak A, \mathfrak B}^0 \subseteq \Gamma^{\delta_\ast}_{\mathfrak A, \mathfrak B}$ (see Notation \ref{n:Gamma}). The following is essentially \cite[Theorem 3.5]{Wei-classext}, but without assuming that $\mathfrak B$ has an approximate identity of projections.

\begin{lemma}\label{l:congruence}
 Let $\mathfrak e : 0 \to \mathfrak B \xrightarrow{\iota} \mathfrak E \xrightarrow{\pi} \mathfrak A \to 0$ be a unital extension of separable $C^\ast$-algebras with $\mathfrak B$ stable, let $\delta_\ast \colon K_\ast(\mathfrak A) \to K_{1-\ast}(\mathfrak B)$ denote the induced boundary map in $K$-theory, and let $x\in K_0(\mathfrak B)$. Then $\Ksix(\mathfrak e) \equiv \Ksix(\mathfrak e \oplus \mathfrak e_x)$ if and only if $x\in \Gamma_{\mathfrak A, \mathfrak B}^{\delta_\ast}$.
\end{lemma}
\begin{proof}
 By Lemmas \ref{l:Advssum}, \ref{l:sucong}, \ref{l:0cong} and \ref{l:moveunit}, $\Ksix(\mathfrak e \oplus \mathfrak e_x)$ is congruent to
\begin{equation}\label{eq:moveunit}
\xymatrix{
K_0(\mathfrak B) \ar[r]^{\iota_0 \quad \qquad} & (K_0(\mathfrak E), [1_{\mathfrak E}] + \iota_0(x)) \ar[r]^{\qquad \pi_0} & (K_0(\mathfrak A), [1_\mathfrak{A}]) \ar[d]^{\delta_0} \\
K_1(\mathfrak A) \ar[u]^{\delta_1} & K_1(\mathfrak E) \ar[l]_{\pi_1} & K_1(\mathfrak B). \ar[l]_{\iota_1} 
}
\end{equation}
If $x\in \Gamma_{\mathfrak A, \mathfrak B}^{\delta_\ast}$, then there is a homomorphism $\phi \colon \ker \delta_0 \to \coker \delta_1$ such that
$q_{\delta_1}(x) = \phi([1_\mathfrak{A}])$. Define $\eta_0 = \id_{K_0(\mathfrak E)} + \overline \iota_0 \circ \phi \circ \pi_0 \colon K_0(\mathfrak E) \to K_0(\mathfrak E)$, where $\overline \iota_0 \colon \coker \delta_1 \to K_0(\mathfrak E)$ is the injective homomorphism 
induced by $\iota_0$. Letting $\eta_1 = \id_{K_1(\mathfrak E)}$ it easily follows that $\eta_\ast \colon K_\ast(\mathfrak E) \to K_\ast(\mathfrak E)$ induces a congruence between $\Ksix(\mathfrak e)$ and the sequence \eqref{eq:moveunit}.

Now suppose that $\Ksix(\mathfrak e)$ is congruent to $\Ksix(\mathfrak e \oplus \mathfrak e_x)$ which in turn is congruent to the sequence \eqref{eq:moveunit}.
There is a homomorphism $\eta_\ast \colon K_\ast(\mathfrak E) \to K_\ast(\mathfrak E)$ such that $\eta_0([1_{\mathfrak E}]) = [1_{\mathfrak E}] + \iota_0(x)$ and the following diagram with exact rows
\[
\xymatrix{
K_1(\mathfrak A) \ar[r]^{\delta_1} \ar@{=}[d] & K_0(\mathfrak B) \ar[r]^{\iota_0} \ar@{=}[d] & K_0(\mathfrak E) \ar[r]^{\pi_0} \ar[d]^{\eta_0} & K_0(\mathfrak A) \ar@{=}[d] \ar[r]^{\delta_0} & K_1(\mathfrak B) \ar@{=}[d]  \\
K_1(\mathfrak A) \ar[r]^{\delta_1} & K_0(\mathfrak B) \ar[r]^{\iota_0} & K_0(\mathfrak E) \ar[r]^{\pi_0} & K_0(\mathfrak A) \ar[r]^{\delta_0} & K_1(\mathfrak B)
}
\]
commutes. By a standard diagram chase, there is a homomorphism $\phi \in \Hom (\ker \delta_0 , \coker \delta_1)$ such that $\eta_0 = \id_{K_0(\mathfrak E)} + \overline \iota_0 \circ \phi \circ \pi_0$, where $\overline \iota_0\colon \coker \delta_1 \to K_0(\mathfrak E)$ is the map induced by $\iota_0$. Hence
\[
[1_{\mathfrak E}] + \iota_0(x) = \eta_0([1_{\mathfrak E}]) = [1_{\mathfrak E}] + \overline \iota_0 \circ \phi ([1_{\mathfrak A}]).
\]
Letting $q_{\delta_1} \colon K_0(\mathfrak B) \to \coker \delta_1$ denote the quotient map, we get $\overline \iota_0( q_{\delta_1}(x)) = \iota_0(x) = \overline\iota_0 \circ \phi([1_{\mathfrak A}])$ which implies $q_{\delta_1}(x) = \phi([1_{\mathfrak{A}}])$ since $\overline \iota_0$ is injective. Thus $x\in \Gamma_{\mathfrak A, \mathfrak B}^{\delta_\ast}$.
\end{proof}

\begin{proposition}\label{p:PV}
Let $\mathfrak A$ be a separable $C^\ast$-algebra satisfying the UCT, and let $\alpha \in \Aut(\mathfrak A)$ be an isomorphism such that $K_\ast(\alpha) = K_\ast(\id_{\mathfrak A})$. Then the induced Pimsner--Voiculescu sequence collapses to a short exact sequence
\begin{equation}\label{eq:PVses}
0 \to K_{1-\ast} (\mathfrak A) \to K_{1-\ast}(\mathfrak A \rtimes_\alpha \mathbb Z) \to K_\ast(\mathfrak A) \to 0,
\end{equation}
and the induced element in $\Ext(K_\ast(\mathfrak A), K_{1-\ast}(\mathfrak A))$ is mapped to 
\[
\kk(\alpha) - \kk(\id_{\mathfrak A}) \in \kk(\mathfrak A, \mathfrak A)
\]
via the map $\Ext(K_\ast(\mathfrak A), K_{1-\ast}(\mathfrak A)) \to \kk(\mathfrak A, \mathfrak A)$ from the UCT.
\end{proposition}
\begin{proof}
That the Pimsner--Voiculescu sequence collapses to a short exact sequence is obvious.

Let $\mathfrak M := \{ f \in C([0,1] , \mathfrak A) : \alpha(f(0)) = f(1)\}$ be the mapping torus of $\alpha$ and $\id_{\mathfrak A}$. It is well-known that the extension
\[
0 \to C_0((0,1), \mathfrak A) \to \mathfrak M \to \mathfrak A \to 0
\]
induces a short exact sequence
\[
0 \to K_{1-\ast}(\mathfrak A) \to K_\ast(\mathfrak M) \to K_\ast(\mathfrak A) \to 0
\]
which represents the element in $\Ext(K_\ast(\mathfrak A) , K_{1-\ast}(\mathfrak A))$ induced by $\kk(\alpha)- \kk(\id_{\mathfrak A})$. By \cite[Section 10.4]{Blackadar-book-K-theory} it follows that this extension is congruent to \eqref{eq:PVses}.
\end{proof}

The following lemma is an immediate consequence of the Elliott--Kucerovsky absorption theorem.

\begin{lemma}\label{l:nucabspullback}
Let $\mathfrak A$ and $\mathfrak C$ be separable, unital, nuclear $C^\ast$-algebras, with a unital embedding $\iota \colon \mathfrak A \to \mathfrak C$, and let $\mathfrak B$ be a separable stable $C^\ast$-algebra. If $\mathfrak e$ is an absorbing, unital extension of $\mathfrak C$ by $\mathfrak B$, then $\mathfrak e \cdot \iota$ is an absorbing, unital extension of $\mathfrak A$ by $\mathfrak B$.
\end{lemma}
\begin{proof}
It follows immediately from the definition of pure largeness that $\mathfrak e \cdot \iota$ is also purely large, so the result follows from \cite[Theorem 6]{ElliottKucerovsky-extensions}.
\end{proof}

In the following, we consider
\[
\Gamma^{(0,\delta_1)}_{\mathfrak A, \mathfrak B} = q_{\delta_1}^{-1}(\{ \psi([1_{\mathfrak A}]) : \psi \in \Hom(K_0(\mathfrak A), \coker \delta_1)\}),
\]
which is a special case of Notation \ref{n:Gammadelta}. Clearly
\[
\Gamma^{(0,\delta_1)}_{\mathfrak A, \mathfrak B} \subseteq \Gamma^{\delta_\ast}_{\mathfrak A, \mathfrak B} \subseteq K_0(\mathfrak B).
\]

The following lemma is the main technical tool to obtain our classification of unital extensions.
While the conditions on $\mathfrak A$ in the following lemma might look slightly technical, we emphasise that any unital UCT Kirchberg algebra has these properties; $K_1$-surjectivity follows from \cite{Cuntz-K-theoryI} and the condition on automorphisms follows from the Kirchberg--Phillips theorem \cite{Kirchberg-simple}, \cite{Phillips-classification}.

\begin{lemma}\label{l:technical}
Let $\mathfrak e : 0 \to \mathfrak B \to \mathfrak E \to \mathfrak A \to 0$ be a unital extension of separable $C^\ast$-algebras with boundary map $\delta_\ast \colon K_\ast(\mathfrak A) \to K_{1-\ast}(\mathfrak B)$ in $K$-theory. Suppose that $\mathfrak B$ is stable, and that $\mathfrak A$ is nuclear, $K_1$-surjective, satisfies the UCT, and that for any $\mathsf y \in \kk(\mathfrak A,\mathfrak A)$ for which $K_\ast(\mathsf y) = K_\ast(\id_{\mathfrak A})$, there is an automorphism $\alpha \in \Aut(\mathfrak A)$ such that $\kk(\alpha) = \mathsf y$. Then for any $x\in \Gamma^{(0, \delta_1)}_{\mathfrak A, \mathfrak B}$ there is an automorphism $\beta \in \Aut(\mathfrak A)$ for which $K_\ast(\beta) = \id_{K_\ast(\mathfrak A)}$, and
\[
[\mathfrak e \cdot \beta]_\s = [\mathfrak e]_\s + [\mathfrak e_x]_\s \in \Ext_\us(\mathfrak A, \mathfrak B).
\]
\end{lemma}
\begin{proof}
Let $\mathfrak e_0$ be an absorbing, trivial, unital extension $\mathfrak e_0$. Since
\[
[(\mathfrak e \oplus \mathfrak e_0) \cdot \beta]_\s = [\mathfrak e \cdot \beta]_\s + [\mathfrak e_0 \cdot \beta]_\s = [\mathfrak e \cdot \beta]_\s
\]
for any automorphism $\beta\in \Aut(\mathfrak A)$, it follows that we may replace $\mathfrak e$ with $\mathfrak e \oplus \mathfrak e_0$ without loss of generality, and thus assume that $\mathfrak e$ is absorbing.

As $x\in \Gamma^{(0,\delta_1)}_{\mathfrak A, \mathfrak B}$ we may find a homomorphism $\psi \colon K_0(\mathfrak A) \to K_0(\mathfrak B)/\im \delta_1$, such that $\psi([1_{\mathfrak A}]) = x + \im \delta_1$.
Let $0 \to F_1 \xrightarrow{f_1} F_0 \xrightarrow{f_0} K_0(\mathfrak A) \to 0$ be a free resolution.\footnote{I.e.~a short exact sequence with both $F_0$ and $F_1$ free abelian groups.} As $F_0$ and $F_1$ are free, we may construct the following commutative diagram with exact rows
\[
\xymatrix{
0 \ar[r] & F_1 \ar[d]_{\psi_1} \ar[r]^{f_1} & F_0 \ar[d]_{\psi_0} \ar[r]^{f_0} & K_0(\mathfrak A) \ar[d]_\psi \ar[r] & 0 \\
& K_1(\mathfrak A) \ar[r]^{\delta_1} & K_0(\mathfrak B) \ar[r] & K_0(\mathfrak B)/\im \delta_1 \ar[r] & 0.
}
\]
Letting $G$ denote the push-out of $\psi_1$ and $f_1$, we get the following commutative diagram
\begin{equation}\label{eq:pushout}
\xymatrix{
0 \ar[r] & F_1 \ar[d]_{\psi_1} \ar[r]^{f_1} & F_0 \ar[d]_{\tilde \psi_0} \ar[r]^{f_0} & K_0(\mathfrak A) \ar@{=}[d] \ar[r] & 0 \\
0 \ar[r] & K_1(\mathfrak A) \ar@{=}[d] \ar[r] & G \ar@{.>}[d]_\phi \ar[r] & K_0(\mathfrak A) \ar[d]_\psi \ar[r] & 0 \\
& K_1(\mathfrak A) \ar[r]^{\delta_1} & K_0(\mathfrak B) \ar[r] & K_0(\mathfrak B)/\im \delta_1 \ar[r] & 0.
}
\end{equation}
with exact rows. The homomorphism $\phi\colon G \to K_0(\mathfrak B)$ making the diagram commute, exists by the universal property of push-outs. Let $\mathsf x_\ast \in \Ext(K_\ast(\mathfrak A), K_{1-\ast}(\mathfrak A)) \subseteq \kk(\mathfrak A, \mathfrak A)$ be such that
\[
\mathsf x_0 = \big[0 \to K_1(\mathfrak A) \to G \to K_0(\mathfrak A) \to 0 \big] \in \Ext(K_0(\mathfrak A), K_1(\mathfrak A)),
\]
and $\mathsf x_1$ is the trivial extension. As $K_\ast(\mathsf x_\ast)$ is the zero map, it follows from our hypothesis on $\mathfrak A$ that there is an automorphism $\alpha \in \Aut(\mathfrak A)$ such that $\kk(\alpha) = \kk(\id_A) + \mathsf x_\ast$.

Applying Proposition \ref{p:PV}, the Pimsner--Voiculescu sequence for the $C^\ast$-dynamical system $(\mathfrak A, \alpha, \mathbb Z)$ collapses to a short exact sequence
\[
0 \to K_{1-\ast}(\mathfrak A) \xrightarrow{\iota_{1-\ast}} K_{1-\ast}(\mathfrak A \rtimes_\alpha \mathbb Z) \to K_\ast(\mathfrak A) \to 0,
\]
which exactly induces the element $\mathsf x_\ast \in \Ext(K_\ast(\mathfrak A), K_{1-\ast}(\mathfrak A))$. Here $\iota \colon \mathfrak A \to \mathfrak A\rtimes_\alpha \mathbb Z$ is the inclusion map. In particular, we may assume that $K_0(\mathfrak A \rtimes_\alpha \mathbb Z) = K_0(\mathfrak A) \oplus K_1(\mathfrak A)$, and $K_1(\mathfrak A \rtimes_\alpha \mathbb Z) = G$, and thus we have a homomorphism
\[
(\delta_0\oplus 0, \phi) \colon K_\ast(\mathfrak A \rtimes_\alpha \mathbb Z) \to K_{1-\ast}(\mathfrak B).
\]
As $\iota_\ast \colon K_\ast(\mathfrak A) \to K_\ast(\mathfrak A \rtimes_\alpha \mathbb Z)$ is injective, it induces a surjection
\[
\iota^\ast \colon \Ext(K_\ast(\mathfrak A \rtimes_\alpha \mathbb Z), K_{\ast}(\mathfrak B)) \twoheadrightarrow \Ext(K_\ast(\mathfrak A), K_{\ast}(\mathfrak B)).
\]
As $\mathfrak A$ satisfies the UCT, so does $\mathfrak A \rtimes_\alpha \mathbb Z$ by \cite{RosenbergSchochet-UCT}. 
Thus, by Theorem \ref{t:sesunitalext}, we get the following commutative diagram
\[
\xymatrix{
\Ext(K_\ast(\mathfrak A\rtimes_\alpha \mathbb Z), K_\ast(\mathfrak B)) \ar@{>->}[r] \ar@{->>}[d]^{\iota^\ast} & \Ext_{\uw}(\mathfrak A \rtimes_\alpha \mathbb Z, \mathfrak B) \ar[d]^{\iota^\ast} \ar@{->>}[r] & \Hom((K_\ast(\mathfrak A \rtimes_\alpha \mathbb Z),[1]), K_{1-\ast}(\mathfrak B)) \ar[d]^{\iota^\ast} \\
\Ext(K_\ast(\mathfrak A), K_\ast(\mathfrak B)) \ar@{>->}[r] & \Ext_{\uw}(\mathfrak A, \mathfrak B) \ar@{->>}[r] & \Hom((K_\ast(\mathfrak A),[1]), K_{1-\ast}(\mathfrak B))
}
\]
for which the rows are short exact sequences.
We may pick $[\mathfrak f']_{\mathrm w} \in \Ext_{\uw}(\mathfrak A \rtimes_\alpha \mathbb Z, \mathfrak B)$ which lifts the homomorphism $(\delta_0 \oplus 0,\phi)$. Recall that we identified $G = K_1(\mathfrak A \rtimes_\alpha \mathbb Z)$, so by \eqref{eq:pushout}, we have
\[
\iota^\ast (\delta_0 \oplus 0, \phi) = ((\delta_0 \oplus 0) \circ \iota_0, \phi \circ \iota_1) = (\delta_0, \delta_1) = \delta_\ast.
\]
Thus $\iota^\ast([\mathfrak f']_{\mathrm w})$ and $[\mathfrak e]_{\mathrm w}$ induce the same element in $\Hom$. Thus, by doing a diagram chase in the above diagram (using surjectivity of the left vertical map), there is an element $[\mathfrak f'']_{\mathrm w} \in \Ext_{\uw}(\mathfrak A \rtimes_\alpha \mathbb Z, \mathfrak B)$ vanishing in $\Hom$, such that $\iota^\ast([\mathfrak f']_{\mathrm w} + [\mathfrak f'']_{\mathrm w}) = [\mathfrak e]_{\mathrm w}$. Let $\mathfrak f$ be an absorbing unital extension of $\mathfrak A \rtimes_\alpha \mathbb Z$ by $\mathfrak B$ such that $[\mathfrak f]_{\mathrm w} = [\mathfrak f']_{\mathrm w} + [\mathfrak f'']_{\mathrm w}$. Then $[\mathfrak f \cdot \iota]_{\mathrm w} = [\mathfrak e]_{\mathrm w}$. 

Let $\tau \colon \mathfrak A \to \corona{\mathfrak B}$ be the Busby map of $\mathfrak e$, and $\eta \colon \mathfrak A \rtimes_\alpha \mathbb Z \to \corona{\mathfrak B}$ be the Busby map of $\mathfrak f$. In particular, $\eta \circ \iota$ is the Busby map of $\mathfrak f \cdot \iota$. Recall from the beginning of the proof that we assumed that $\mathfrak e$ was absorbing, and by Lemma \ref{l:nucabspullback}, $\mathfrak f \cdot \iota$ is also absorbing. Thus, as $[\mathfrak f \cdot \iota]_w = [\mathfrak e]_w$, there is a unitary $u\in \corona{\mathfrak B}$ such that
\[
\Ad u \circ \tau = \eta \circ \iota.
\]
Let $w\in \mathfrak A \rtimes_\alpha \mathbb Z$ denote the canonical unitary, so that $\Ad w \circ \iota = \iota \circ \alpha$. Then
\begin{eqnarray*}
\tau \circ \alpha &=& \Ad u^\ast \circ \eta \circ \iota \circ \alpha \\
&=& \Ad u^\ast \circ \eta \circ \Ad w \circ \iota \\
&=& \Ad u^\ast \circ \Ad \eta(w) \circ \eta \circ \iota \\
&=& \Ad u^\ast \circ \Ad \eta(w) \circ \Ad u \circ \tau \\
&=& \Ad (u^\ast \eta(w) u) \circ \tau.
\end{eqnarray*}
Hence it follows from Lemma \ref{l:Advssum} that 
\[
[\mathfrak e \cdot \alpha]_\s = [\mathfrak e]_\s + [\mathfrak e_{[u^\ast \eta(w) u]}]_\s = [\mathfrak e]_\s + [\mathfrak e_{[\eta(w)]}]_\s \in \Ext_\us(\mathfrak A, \mathfrak B).
\]
Recall that $[\mathfrak f]_\w = [\mathfrak f']_\w + [\mathfrak f'']_\w$ where $[\mathfrak f'']_\w$ vanishes in $\Hom$, and $[\mathfrak f']_\w$ induces the homomorphism $(\delta_0\oplus 0,\phi) \colon K_\ast(\mathfrak A \rtimes_{\alpha} \mathbb Z) \to K_{1-\ast}(\mathfrak B)$. Thus $[\mathfrak f]_\w$ also induces the homomorphism $(\delta_0 \oplus 0,\phi)$, so in particular 
\[
K_1(\eta) = \phi \colon K_1(\mathfrak A \rtimes_\alpha \mathbb Z) \to K_1(\corona{\mathfrak B}) = K_0(\mathfrak B).
\]
 It follows that
\begin{equation}\label{eq:alphaext}
[\mathfrak e \cdot \alpha]_\s = [\mathfrak e]_\s + [\mathfrak e_{\phi([w])}]_\s \in \Ext_\us(\mathfrak A, \mathfrak B).
\end{equation}
By commutativity of the lower right square in \eqref{eq:pushout}, the two compositions 
\[
K_1(\mathfrak A \rtimes_\alpha \mathbb Z) \xrightarrow \phi K_0(\mathfrak B) \to K_0(\mathfrak B)/\im \delta_1, \qquad K_1(\mathfrak A \rtimes_\alpha \mathbb Z) \to K_0(\mathfrak A) \xrightarrow \psi K_0(\mathfrak B)/\im \delta_1,
\]
are the same. It is well-known, that $[w]$ is mapped to $[1_{\mathfrak A}]$ via the map $K_1(\mathfrak A \rtimes_\alpha \mathbb Z) \to K_0(\mathfrak A)$.\footnote{The proof of this is identical to the proof showing that the map $K_1(C(\mathbb T)) \to K_0(\mathbb K)$ induced by the usual Toeplitz extension, sends the class of the canonical unitary in $C(\mathbb T)$ to $[e_{11}]_0$.} Thus
\[
\phi([w]) + \im \delta_1 = \psi([1_\mathfrak{A}]) = x + \im \delta_1,
\]
where $x\in \Gamma^{(0,\delta_1)}_{\mathfrak A, \mathfrak B}$ is our given element from the statement of the lemma. As $\mathfrak A$ is $K_1$-surjective we may find a unitary $v\in \mathfrak A$ such that 
\begin{equation}\label{eq:phidelta}
\phi([w]) + \delta_1([v]) = x.
\end{equation}
Let $\beta = \Ad v \circ \alpha$ be the induced automorphism on $\mathfrak A$. By construction $K_\ast(\alpha) = \id_{K_\ast(\mathfrak A)}$ and thus $K_\ast(\beta) = \id_{K_\ast(\mathfrak A)}$. By Lemma \ref{l:Advssum} it follows that
\begin{eqnarray*}
[\mathfrak e \cdot \beta]_\s &=& [\mathfrak e \cdot \alpha]_\s + [\mathfrak e_{\delta_1([v]_1)}]_\s \\
&\stackrel{\eqref{eq:alphaext}}{=}& [\mathfrak e]_\s + [\mathfrak e_{\phi([w]_1)}]_\s + [\mathfrak e_{\delta_1([v]_1)}]_\s \\
&=& [\mathfrak e]_\s + [\mathfrak e_{\phi([w]_1) + \delta_1([v]_1)}]_\s \\
&\stackrel{\eqref{eq:phidelta}}{=}& [\mathfrak e]_\s + [\mathfrak e_x]_\s
\end{eqnarray*}
as desired.
\end{proof}

\begin{proposition}\label{p:mainprop}
Let $\mathfrak A$ and $\mathfrak B$ be separable $C^\ast$-algebras, with $\mathfrak A$ unital, nuclear and satisfying the UCT, and $\mathfrak B$ stable. Suppose that $\mathfrak A$ is $K_1$-surjective and that for any $\mathsf y \in \kk(\mathfrak A,\mathfrak A)$ for which $K_\ast(\mathsf y) = K_\ast(\id_{\mathfrak A})$, there is an automorphism $\alpha \in \Aut(\mathfrak A)$ such that $\kk(\alpha) = \mathsf y$. Let $\mathfrak e_1$ and $\mathfrak e_2$ be unital extensions of $\mathfrak A$ by $\mathfrak B$ and suppose that
\begin{itemize}
\item[$(a)$] $[\mathfrak e_1]_\mathrm{w} = [\mathfrak e_2]_\mathrm{w}$ in $\Ext_{\uw}(\mathfrak A, \mathfrak B)$,
\item[$(b)$] $\Ksix(\mathfrak e_1) \equiv \Ksix(\mathfrak e_2)$,
\item[$(c)$] the exponential maps $\delta_0 \colon K_0(\mathfrak A) \to K_{1-\ast}(\mathfrak B)$ induced by $\mathfrak e_1$ and $\mathfrak e_2$ vanish.
\end{itemize} 
Then there is an automorphism $\beta \in \Aut(\mathfrak A)$ with $K_\ast(\beta) = \id_{K_\ast(\mathfrak A)}$ such that $[\mathfrak e_1 \cdot \beta] = [\mathfrak e_2]$ in $\Ext_{\us}(\mathfrak A, \mathfrak B)$.
\end{proposition}
\begin{proof}
Let $\delta_\ast \colon K_\ast(\mathfrak A) \to K_{1-\ast}(\mathfrak B)$ be the connecting maps in the six-term exact sequences of $\mathfrak e_1$ and $\mathfrak e_2$, which agree since $\Ksix(\mathfrak e_1) \equiv \Ksix(\mathfrak e_2)$. As $[\mathfrak e_1]_\mathrm{w} = [\mathfrak e_2]_\mathrm{w}$, it follows from Theorem \ref{t:sesunitalext} that there is an $x\in K_0(\mathfrak B)$ such that $[\mathfrak e_1\oplus\mathfrak e_x]_\s = [\mathfrak e_2]_\s$ in $\Ext_\us(\mathfrak A, \mathfrak B)$. As 
\[
\Ksix(\mathfrak e_1 \oplus \mathfrak e_x) \stackrel{\textrm{Cor.}~\ref{c:KsixExt}}{\equiv} \Ksix(\mathfrak e_2) \equiv \Ksix(\mathfrak e_1),
\]
 it follows from Lemma \ref{l:congruence} that $x\in \Gamma^{\delta_\ast}_{\mathfrak A, \mathfrak B}$. As $\delta_0 =0$, it clearly holds that $\Gamma^{\delta_\ast}_{\mathfrak A, \mathfrak B} = \Gamma^{(0,\delta_1)}_{\mathfrak A, \mathfrak B}$ and thus Lemma \ref{l:technical} provides an automorphism $\beta\in \Aut(\mathfrak A)$ such that
\[
[\mathfrak e_1 \cdot \beta]_\s = [\mathfrak e_1]_\s + [\mathfrak e_x]_\s = [\mathfrak e_2]_\s \in \Ext_\us(\mathfrak A, \mathfrak B)
\]
as wanted.
\end{proof}

\begin{remark}
The only thing Condition $(c)$ was used for in Proposition \ref{p:mainprop} was so that $\Gamma^{(0, \delta_1)}_{\mathfrak A, \mathfrak B} = \Gamma^{\delta_\ast}_{\mathfrak A, \mathfrak B}$. Hence one may replace Condition $(c)$ with this more general condition in order to obtain the conclusion of Proposition \ref{p:mainprop}. 

In particular, Condition $(c)$ in Proposition \ref{p:mainprop} may be replaced by any of the following statements, as these all imply that $\Gamma^{(0, \delta_1)}_{\mathfrak A, \mathfrak B} = \Gamma^{\delta_\ast}_{\mathfrak A, \mathfrak B}$. Proving that $(c1)$--$(c6)$ imply $\Gamma^{(0,\delta_1)}_{\mathfrak A, \mathfrak B} = \Gamma^{\delta_\ast}_{\mathfrak A, \mathfrak B}$ is left to the reader.
\begin{itemize}
\item[$(c1)$] The class of the unit $[1_{\mathfrak A}]$ vanishes in $K_0(\mathfrak A)$.
\item[$(c2)$] The exponential map $\delta_0$ is injective.
\item[$(c3)$] The index map $\delta_1$ is surjective.
\item[$(c4)$] $K_0(\mathfrak A) \cong \mathbb Z \oplus G$, such that $[1_{\mathfrak A}] = (1,g)$ for some $g\in G$.
\item[$(c5)$] $\ker \delta_0$ is a direct summand in $K_0(\mathfrak A)$.
\item[$(c6)$] $K_0(\mathfrak E)$ is divisible.
\end{itemize}
\end{remark}

\begin{proposition}\label{p:congtoExt}
Let $\mathfrak e_i \colon 0 \to \mathfrak B \to \mathfrak E_i \to \mathfrak A \to 0$ be unital extensions of $C^\ast$-algebras for $i=1,2$ such that $\mathfrak A$ is a unital UCT Kirchberg algebra, and $\mathfrak B$ is a stable AF algebra. If $\Ksix(\mathfrak e_1) \equiv \Ksix(\mathfrak e_2)$ then there is an automorphism $\alpha \in \Aut(\mathfrak A)$ such that $\mathfrak e_1$ and $\mathfrak e_2 \cdot \alpha$ are strongly unitarily equivalent.

In particular, if $\Ksix(\mathfrak e_1) \equiv \Ksix(\mathfrak e_2)$ then $\mathfrak E_1 \cong \mathfrak E_2$.
\end{proposition}
\begin{proof}
We identify $\Ext(\mathfrak A, \mathfrak B) \cong \kk^{1} ( \mathfrak{A} , \mathfrak{B} )$ in the usual way, see Theorem \ref{t:KKExt}.  By \cite[Theorem 2.3]{EilersRestorffRuiz-classext} (which is based on \cite[Theorem 3.2]{Rordam-classsixterm}), there exist $\mathsf x \in \kk ( \mathfrak{A} , \mathfrak{A} )$ and $\mathsf y \in \kk(\mathfrak{B} , \mathfrak{B} )$ such that $K_* ( \mathsf x ) = K_* ( \mathrm{id}_\mathfrak{A} )$, $K_* ( \mathsf y ) = K_* ( \mathrm{id}_\mathfrak{B} )$, $[ \mathfrak{e}_1 ] \times \mathsf y = \mathsf x \times [ \mathfrak{e}_2 ]$ in $\kk^1 ( \mathfrak{A} , \mathfrak{B} )$.  Since $\mathfrak{B}$ is an AF algebra, we have that $\mathsf y = \kk ( \mathrm{id}_\mathfrak{B} )$.  Thus $[ \mathfrak{e}_1 ] = \mathsf x \times [ \mathfrak{e}_2 ]$.  Since $\mathfrak{A}$ is a UCT Kirchberg algebra, by the Kirchberg--Phillips theorem \cite{Kirchberg-simple}, \cite{Phillips-classification} there exists an isomorphism $\alpha^{(1)} \colon \mathfrak{A} \to \mathfrak{A}$ such that $\mathsf x = \kk ( \alpha^{(1)} )$.   By \cite[Proposition 1.1]{Rordam-classsixterm}, we get 
\[
 [\mathfrak{e}_2 \cdot \alpha^{(1)}] = \mathsf x \times [ \mathfrak{e}_2 ] = [\mathfrak e_1] \in \kk^1(\mathfrak A, \mathfrak B) \cong \Ext(\mathfrak A, \mathfrak B).
\]
By Lemma \ref{l:UCT} it follows that $[\mathfrak e_1]_\w = [\mathfrak e_2 \cdot \alpha^{(1)}]_\w$ in $\Ext_\uw(\mathfrak A, \mathfrak B)$. Now, as $K_\ast(\alpha^{(1)}) = \id_{K_\ast(\mathfrak A)}$, it follows that 
\[
\Ksix(\mathfrak e_2 \cdot \alpha^{(1)}) \equiv \Ksix(\mathfrak e_2)  \equiv \Ksix(\mathfrak e_1).
\]
By \cite{Cuntz-K-theoryI} $\mathfrak A$ is $K_1$-surjective, and by the Kirchberg--Phillips theorem (cited above) $\mathfrak A$ satisfies the condition in Proposition \ref{p:mainprop} about automorphisms. Hence this proposition produces an automorphism $\alpha^{(2)}\in \Aut(\mathfrak A)$ with $K_\ast(\alpha^{(2)}) = \id_{K_\ast(\mathfrak A)}$ such that
\[
[\mathfrak e_1]_\s = [\mathfrak e_2 \cdot \alpha]_\s \in \Ext_\us(\mathfrak A, \mathfrak B)
\]
where $\alpha = \alpha^{(1)} \circ \alpha^{(2)}$.

As $\mathfrak A$ is simple, unital and nuclear, $\mathfrak B$ is stable with the corona factorisation property, and the extensions $\mathfrak e_1$ and $\mathfrak e_2 \cdot \alpha$ are unital, it follows that $\mathfrak e_1$ and $\mathfrak e_2 \cdot \alpha$ are full and thus absorbing. Hence $\mathfrak e_1$ and $\mathfrak e_2 \cdot \alpha$ are strongly unitarily equivalent.

The ``in particular'' part follows since the extension algebra of $\mathfrak e_2$ is isomorphic to the extension algebra of $\mathfrak e_2 \cdot \alpha$ and since strong unitary equivalence implies isomorphism of the extension algebras.
\end{proof}

By $\Ksixpos(\mathfrak e)$ we mean the six-term exact sequence in $K$-theory with order in all $K_0$-groups. The following is the main classification result of this section and is Theorem \ref{t:classunital}.

\begin{theorem}
Let $\mathfrak e_i : 0 \to \mathfrak B_i \to \mathfrak E_i \to \mathfrak A_i \to 0$ be unital extensions of $C^\ast$-algebras for $i=1,2$, such that $\mathfrak A_1$ and $\mathfrak A_2$ are unital UCT Kirchberg algebras and $\mathfrak B_1$ and $\mathfrak B_2$ are stable AF algebras. Then $\mathfrak E_1 \cong \mathfrak E_2$ if and only if $\Ksixpos(\mathfrak e_1) \cong \Ksixpos(\mathfrak e_2)$.
\end{theorem}
\begin{proof}
Suppose $\mathfrak E_1 \cong \mathfrak E_2$. As the extension $\mathfrak e_i$ is unital, and as $\mathfrak A_i$ is simple, it follows that the extension $\mathfrak e_i$ is full. As $\mathfrak B_i$ is stable, it therefore follows that $\mathfrak B_i$ is the unique maximal ideal in $\mathfrak E_i$ for $i=1,2$.\footnote{\label{fn}Clearly $\mathfrak B_1$ is a maximal ideal as the corresponding quotient is simple. If $\mathfrak J \subseteq \mathfrak E_1$ is a two-sided, closed ideal such that $\mathfrak J \not \subseteq \mathfrak B_1$, then there is an element $x\in \mathfrak J \setminus \mathfrak B_1$ inducing a non-zero element in $\mathfrak A_1$. As the extension is full and $\mathfrak B_1$ is stable, it follows that $x$ induces a full element in $\multialg{\mathfrak B_1}$. Hence $\overline{ \mathfrak B_1 x \mathfrak B_1} = \mathfrak B_1$ so $\mathfrak B_1 \subsetneq \mathfrak J$ and thus $\mathfrak J= \mathfrak E_1$ by maximality of $\mathfrak B_1$. The same argument works for $\mathfrak E_2$.} It follows that the extensions $\mathfrak e_1$ and $\mathfrak e_2$ are isomorphic, and thus $\Ksixpos(\mathfrak e_1) \cong \Ksixpos(\mathfrak e_2)$.

Now suppose that there is an isomorphism $\Ksixpos(\mathfrak e_1) \xrightarrow \cong \Ksixpos(\mathfrak e_2)$ induced by
\[
\phi_\ast \colon K_\ast(\mathfrak A_1) \xrightarrow \cong K_\ast(\mathfrak A_2), \quad \psi_\ast\colon K_\ast^{+}(\mathfrak B_1) \xrightarrow \cong K_\ast^+(\mathfrak B_2) , \quad \rho_\ast \colon K_\ast(\mathfrak E_1) \xrightarrow \cong K_\ast(\mathfrak E_2).
\]
By the Kirchberg--Phillips theorem \cite{Kirchberg-simple}, \cite{Phillips-classification} we find an isomorphism $\alpha \colon \mathfrak A_1 \xrightarrow \cong \mathfrak A_2$ such that $K_\ast(\alpha) = \phi_\ast$. Similarly, by Elliott's classification of AF algebras \cite{Elliott-AFclass}, we find an isomorphism $\beta \colon \mathfrak B_1 \xrightarrow \cong \mathfrak B_2$ such that $K_\ast(\beta) = \psi_\ast$. We obtain the following commutative diagram
\begin{equation}\label{eq:bigpbpo}
\xymatrix{
\mathfrak e_1 : & 0 \ar[r] & \mathfrak B_1 \ar[d]^\beta_\cong \ar[r] & \mathfrak E_1 \ar[d]_\cong^{\eta^{(1)}} \ar[r] & \mathfrak A_1 \ar@{=}[d] \ar[r] & 0 \\
\beta \cdot \mathfrak e_1 : & 0 \ar[r] & \mathfrak B_2 \ar@{=}[d] \ar[r] & \mathfrak E_1' \ar[r] & \mathfrak A_1 \ar@{=}[d] \ar[r] & 0 \\
\mathfrak e_2 \cdot \alpha : & 0 \ar[r] & \mathfrak B_2 \ar@{=}[d] \ar[r] & \mathfrak E_2' \ar[r] \ar[d]_\cong^{\eta^{(2)}} & \mathfrak A_1 \ar[d]^\alpha_\cong \ar[r] & 0 \\
\mathfrak e_2 : & 0 \ar[r] & \mathfrak B_2 \ar[r] & \mathfrak E_2 \ar[r] & \mathfrak A_2 \ar[r] & 0
}
\end{equation}
which has exact rows. It is easy to see that the map
\[
K_\ast(\eta^{(2)})^{-1} \circ \rho_\ast \circ K_\ast(\eta^{(1)})^{-1} \colon K_\ast(\mathfrak E_1') \to K_\ast(\mathfrak E_2')
\]
induces a congruence $\Ksix(\beta \cdot \mathfrak e_1) \equiv \Ksix(\mathfrak e_2 \cdot \alpha)$. By Proposition \ref{p:congtoExt} it follows that $\mathfrak E_1' \cong \mathfrak E_2'$, so it follows that $\mathfrak E_1 \cong \mathfrak E_2$.
\end{proof}

%%%%%%%%%%%%%%%%%%%%%%%%%%%%%%%%%%%%%%%%%%%%%%%%%%%%%%%%%
%%%%%%%%%%%%%%%%%%%%%%%%%%%%%%%%%%%%%%%%%%%%%%%%%%%%%%%%%

\section{Determining when extensions are full}

In this section we characterise when certain extensions are full with a stable ideal. We show that when the ideal is sufficiently finite (e.g.~an AF algebra) and the quotient is sufficiently infinite (e.g.~a Kirchberg algebra), then this is characterised by the existence of a properly infinite, full projection in the extension algebra.

\begin{lemma}\label{l:stableAF}
Let $\mathfrak B$ be a $\sigma$-unital $C^\ast$-algebra with stable rank one. Then $\mathfrak B$ is stable if and only if there exists a projection $p \in \multialg{\mathfrak B}$ which is properly infinite, and which is strictly full, i.e.~$\overline{\mathfrak B p \mathfrak B} = \mathfrak B$.

In particular, if $p\in \multialg{\mathfrak B}$ is a strictly full, properly infinite projection, then $p\mathfrak B p$ is stable.
\end{lemma}
\begin{proof}
If $\mathfrak B$ is stable then $1_{\multialg{\mathfrak B}} \in \multialg{\mathfrak B}$ is a strictly full, properly infinite projection.

Conversely, suppose $p\in \multialg{\mathfrak B}$ is a strictly full, properly infinite projection.
Let $p_1,p_2, \dots \in \multialg{\mathfrak B}$ be a sequence of pairwise orthogonal projections in $\multialg{\mathfrak B}$, such that $p_i \leq p$ and $p\sim p_i$ for all $i\in \mathbb N$. Then the hereditary $C^\ast$-subalgebra $\mathfrak B_0$ of $\mathfrak B$ generated by $p_1,p_2,\dots$ is isomorphic to $p\mathfrak B p\otimes \mathbb K$. As $p$ is strictly full it follows that $\mathfrak B_0 \subseteq \mathfrak B$  is a stable, full, hereditary $C^\ast$-subalgebra. It is an easy consequence of \cite[Lemma 4.6]{PereraTomsWhiteWinter-Cu} that $\mathfrak B$ is stable (as any strictly positive element in $\mathfrak B_0$ induces a full, properly infinite element in the scale of the Cuntz semigroup of $\mathfrak B$).

``In particular'' is immediate since $\multialg{p\mathfrak B p} \cong p\multialg{\mathfrak B} p$ canonically, and since $p\mathfrak B p$ is $\sigma$-unital with stable rank one.
\end{proof}

The following is essentially \cite[Proposition 2.7]{BlanchardRohdeRordam-K1inj}.

\begin{lemma}\label{l:propinfpullback}
Let $\mathfrak A, \mathfrak C$ and $\mathfrak D$ be $C^\ast$-algebras and suppose that $\phi \colon \mathfrak A \to \mathfrak D$ and $\pi \colon \mathfrak C \to \mathfrak D$ are $\ast$-homomorphisms for which $\pi$ is surjective. Suppose that $p \in \mathfrak A$ and $q\in \mathfrak C$ are projections such that $\phi(p) = \pi(q)$ and $\phi(p) \mathfrak D \phi(p)$ is $K_1$-injective. If both $p$ and $q$ are properly infinite, then $p\oplus q$ is properly infinite in the pull-back $\mathfrak A \oplus_{\phi,\pi} \mathfrak C$. 
\end{lemma}
\begin{proof}
By replacing $\mathfrak A$, $\mathfrak C$ and $\mathfrak D$ with $p\mathfrak A p $, $q \mathfrak C q$ and $\phi(p) \mathfrak D \phi(p)$, we may assume that $\mathfrak A$, $\mathfrak C$ and $\mathfrak D$ are unital and properly infinite, that $\phi$ and $\pi$ are unital maps, and that $\mathfrak D$ is $K_1$-injective. Under these assumptions, we should show that $\mathfrak A \oplus_{\phi, \pi}\mathfrak C$ is properly infinite.

The result now follows from \cite[Proposition 2.7]{BlanchardRohdeRordam-K1inj}. In fact, although said result assumes that both maps are surjective (corresponding in our case to $\phi$ and $\pi$), they only use that one map is surjective. 
We fill in the proof for completion.

Let $s_1,s_2,s_3\in \mathfrak A$ and $t_1, t_2, t_3 \in \mathfrak C$ be isometries with mutually orthogonal range projections. Let
\[
v := \sum_{j=1}^2 \phi (s_j) \pi (t_j)^\ast \in \mathfrak D
\]
which is a partial isometry satisfying $\phi (s_j) = v \pi  (t_j)$ for $j=1,2$. Note that
\[
1_{\mathfrak D} \sim \phi (s_3s_3^\ast) \leq 1_{\mathfrak D}- vv^\ast ,\qquad 1_{\mathfrak D}\sim \pi (t_3t_3^\ast) \leq 1_{\mathfrak D}- v^\ast v.
\]
It follows that $1_{\mathfrak D}- vv^\ast$ and $1_{\mathfrak D}-v^\ast v$ are properly infinite and full in $\mathfrak D$. By \cite[Lemma 2.4$(i)$]{BlanchardRohdeRordam-K1inj} there is a unitary $u\in \mathfrak D$ with $[u]=0\in K_1(\mathfrak D)$ such that $v = u v^\ast v$. As $\mathfrak D$ is $K_1$-injective, it follows that $u$ is homotopic to $1$, and thus lifts to a unitary $\widetilde u\in \mathfrak C$.

Clearly $\widetilde u t_1, \widetilde u t_2 \in \mathfrak C$ are isometries with orthogonal range projections, and
\[
\pi(\widetilde u t_j) = u \pi(t_j) = v \pi(t_j) = \phi (s_j)
\]
so $s_j \oplus \widetilde u t_j \in \mathfrak A \oplus_{\phi, \pi} \mathfrak C$ for $j=1,2$ are isometries with orthogonal range projections. Hence $\mathfrak A\oplus_{\phi, \pi} \mathfrak C$ is properly infinite.
\end{proof}

By the above lemma we deduce the following property about proper infiniteness of projections in purely large extensions (see Remark \ref{r:purelylarge}).

\begin{proposition}\label{p:propinfext}
Let $0 \to \mathfrak B \to \mathfrak E \to \mathfrak A \to 0$ be a purely large extension of separable $C^\ast$-algebras such that $\mathfrak B$ is stable, and suppose that $p\in \mathfrak E \setminus \mathfrak B$ is a projection. Then $p$ is properly infinite if and only if $p+ \mathfrak B \in \mathfrak A$ is properly infinite.
\end{proposition}
\begin{proof}
``Only if'' is trivial. To prove ``if'', assume that the image of $p$ in $\mathfrak A$ is properly infinite. Let $\tau \colon \mathfrak A \to \corona{\mathfrak B}$ be the Busby map of the extension. We may identify $\mathfrak E$ with the pull-back $\mathfrak A \oplus_{\tau, \pi_{\mathfrak B}} \multialg{\mathfrak B}$. Let $q\in \multialg{\mathfrak B}$ be the projection induced by $p$. As purely large extensions are full,\footnote{It is easy to see that an extension $\mathfrak e$ is full if and only if the Cuntz sum $\mathfrak e \oplus 0$ is full. If $\mathfrak e$ is purely large, then $\mathfrak e \oplus 0$ is nuclearly absorbing by \cite[Corollary 2.4]{Gabe-nonunitalext}. As $\mathfrak e \oplus 0$ absorbs any full, trivial, weakly nuclear extension (which always exist), it follows that $\mathfrak e \oplus 0$ -- and thus also $\mathfrak e$ -- is full.} it follows that $q$ is full in $\multialg{\mathfrak B}$. As our given extension is purely large it easily follows that the extension
\[
0 \to \mathfrak B \to \mathfrak B + \mathbb C q \to \mathbb C \to 0
\]
is purely large. By \cite[Proposition 2.7]{Gabe-nonunitalext} it follows that $q$ is a properly infinite, full projection in $\multialg{\mathfrak B}$. Hence $q \mathfrak B q \cong \mathfrak B$ is stable, and thus $\pi_{\mathfrak B} (q) \corona{\mathfrak B} \pi_\mathfrak{B} (q) \cong \corona{\mathfrak B}$ is $K_1$-injective by Proposition \ref{p:K1corona}. By Lemma \ref{l:propinfpullback}, $p$ is properly infinite.
\end{proof}

\begin{proposition}\label{p:propinfvsfull}
Let $\mathfrak e : 0 \to \mathfrak B \to \mathfrak E \to \mathfrak A \to 0$ be an extension of separable $C^\ast$-algebras for which $\mathfrak A$ is simple and $\mathfrak B$ has stable rank one and the corona factorisation property. Suppose that there is a projection $p \in \mathfrak E \setminus \mathfrak B$ such that $p + \mathfrak B \in \mathfrak A$ is properly infinite. Then $\mathfrak B$ is stable and $\mathfrak e$ is full if and only if $p$ is full and properly infinite in $\mathfrak E$.
\end{proposition}
\begin{proof}
``Only if'' follows from Proposition \ref{p:propinfext} as $\mathfrak e$ is purely large by the corona factorisation property. For ``if'' suppose that $p$ is full and properly infinite. Then $\mathfrak B = \overline{\mathfrak B p \mathfrak B}$ by fullness of $p$. By Lemma \ref{l:stableAF} it follows that $\mathfrak B$ is stable, and $\mathfrak B \cong p \mathfrak B p$. Hence by \cite[Theorem 4.23]{Brown-semicontmultipliers}, $p$ induces a full projection in $\multialg{\mathfrak B}$. As $\mathfrak A$ is simple, and as $p + \mathfrak B \in \mathfrak A$ is mapped to a full projection in $\corona{\mathfrak B}$ via the Busby map, it follows that the extension $\mathfrak e$ is full.
\end{proof}

The following can be used to characterise when the extensions we wish to classify are full. 

\begin{theorem}\label{t:fullext}
Let $\mathfrak e : 0 \to \mathfrak B \to \mathfrak E \to \mathfrak A \to 0$ be an extension of $C^\ast$-algebras such that $\mathfrak A$ is a Kirchberg algebra and $\mathfrak B$ is an AF algebra. The following are equivalent.
\begin{itemize}
\item[$(i)$] $\mathfrak B$ is stable and the extension $\mathfrak e$ is full,
\item[$(ii)$] $\mathfrak E$ contains a full, properly infinite projection,
\item[$(iii)$] any projection $p\in \mathfrak E \setminus \mathfrak B$ is full and properly infinite (in $\mathfrak E$).
\end{itemize}
\end{theorem}
\begin{proof}
$(i) \Rightarrow (iii)$: Suppose that $p\in \mathfrak E \setminus \mathfrak B$ is a projection. Fullness of $\mathfrak e$ and simplicity of $\mathfrak A$ imply that $p$ is full. As $\mathfrak B$ has the corona factorisation property by virtue of being an AF algebra, it follows from Proposition \ref{p:propinfext} that $p$ is properly infinite.

$(iii) \Rightarrow (ii)$: Let $q\in \mathfrak A$ be a non-zero projection. By \cite[Proposition 3.15]{BrownPedersen-rr0}, $q$ lifts to a projection $p\in \mathfrak E \setminus \mathfrak B$, which is properly infinite and full by assumption.

$(ii) \Rightarrow (i)$: Follows from Proposition \ref{p:propinfvsfull}
\end{proof}

%%%%%%%%%%%%%%%%%%%%%%%%%%%%%%%%%%%%%%%%%%%%%%%%%%%%%%%%%
%%%%%%%%%%%%%%%%%%%%%%%%%%%%%%%%%%%%%%%%%%%%%%%%%%%%%%%%%

\section{Classification of non-unital extensions}\label{s:nonunital}

In \cite[Section 4]{Gabe-nonunitalext} an example was given of two non-unital, full extensions $\mathfrak e_i : 0 \to \mathfrak B_i \to \mathfrak E_i \to \mathfrak A_i  \to 0$ such that $\mathfrak A_i \cong \mathcal O_2$, $\mathfrak B_i \cong M_{2^\infty} \otimes \mathbb K$, $\Ksixnu(\mathfrak e_1) \cong \Ksixnu(\mathfrak e_2)$ (with order, scale and units preserved), but for which $\mathfrak E_1 \not \cong \mathfrak E_2$. In this section we will describe how to obtain classification of such (and more general) extensions. Note that our invariant needs to carry more information than the six-term exact sequence alone.

The following lemma indicates the main trick that will be used to get classification of non-unital extensions with unital quotients. It implies that if one can arrange that the corresponding Busby maps have the same unit, and that the units in the quotients lift to projections, then the classification problem can be reduced to the unital case.

\begin{lemma}\label{l:unitvsnonunit}
Let $\mathfrak A$ and $\mathfrak B$ be $C^\ast$-algebras with $\mathfrak A$ unital, and let $\tau_i \colon \mathfrak A \to \corona{\mathfrak B}$ be (not necessarily unital) Busby maps for $i=1,2$. Suppose that $\tau_1(1_\mathfrak{A}) = \tau_2(1_{\mathfrak A})$, and that this projection lifts to a projection $p\in \multialg{\mathfrak B}$. If the unital extensions
\begin{equation}\label{eq:cutdownext}
0 \to p \mathfrak B p \to (1_{\mathfrak A}\oplus p)(\mathfrak A \oplus_{\tau_i, \pi_{\mathfrak B}} \multialg{\mathfrak B} ) (1_{\mathfrak A} \oplus p) \to \mathfrak A \to 0
\end{equation}
for $i=1,2$ are strongly unitarily equivalent, then so are the extensions induced by $\tau_1$ and $\tau_2$.
\end{lemma}
\begin{proof}
The Busby maps $\widetilde \tau_i$ of the extensions \eqref{eq:cutdownext} are just the corestrictions of the Busby maps $\tau_i$ to $\tau_i(1_{\mathfrak A}) \corona{B} \tau_i(1_{\mathfrak A}) \cong \corona{p\mathfrak B p}$ (the canonical isomorphism). By assumption there is a unitary $\widetilde u\in \multialg{p\mathfrak B p}$ such that $\Ad \pi_{p \mathfrak B p}(\widetilde u) \circ \widetilde \tau_1 = \widetilde \tau_2$. Using the canonical identification $\multialg{p\mathfrak B p} \cong p \multialg{\mathfrak B} p$, let $u = \widetilde u + (1_{\multialg{B}} - p)$. Then $u$ is a unitary in $\multialg{\mathfrak B}$ satisfying $\Ad \pi_{\mathfrak B}(u) \circ \tau_1 = \tau_2$.
\end{proof}

The next goal will be to arrange that $\tau_1 (1_\mathfrak{A}) = \tau_2(1_{\mathfrak A}) \in \corona{\mathfrak B}$ by twisting one extension by an automorphism on $\mathfrak B$. For this we introduce the following notation.

\begin{notation}\label{n:De}
Let $\mathfrak e  : 0 \to \mathfrak B \to \mathfrak E \xrightarrow \pi \mathfrak A \to 0$ be an extension of $C^\ast$-algebras where $\mathfrak A$ is unital, but $\mathfrak E$ is not necessarily unital. Let 
\[
\mathfrak D_{\mathfrak e} := \pi^{-1}(\mathbb C 1_\mathfrak A) \subseteq \mathfrak E.
\]
In the case where $\mathfrak E$ is unital, then $\mathfrak D = \widetilde{\mathfrak B}$ is the (forced) unitisation of $\mathfrak B$. 
\end{notation}

\begin{lemma}\label{l:De}
Let $\mathfrak e_i : 0 \to \mathfrak B_i \to \mathfrak E_i \to \mathfrak A_i \to 0$ be extensions of $C^\ast$-algebras for $i=1,2$ with Busby maps $\tau_i \colon \mathfrak A_i \to \corona{\mathfrak B_i}$. Suppose that $\mathfrak A_1$ and $\mathfrak A_2$ are unital, that $\beta \colon \mathfrak B_1 \xrightarrow \cong \mathfrak B_2$ is an isomorphism, and let $\overline \beta \colon \corona{\mathfrak B_1} \xrightarrow \cong \corona{\mathfrak B_2}$ be the induced isomorphism of corona algebras. Then $\overline \beta \circ \tau_1(1_{\mathfrak A_1}) = \tau_2(1_{\mathfrak A_2})$ if and only if there is a $\ast$-homomorphism $\mu \colon \mathfrak D_{\mathfrak e_1} \to \mathfrak D_{\mathfrak e_2}$ (necessarily unique and necessarily an isomorphism) making the diagram
\[
\xymatrix{
0 \ar[r] & \mathfrak B_1 \ar[d]^\beta \ar[r] & \mathfrak D_{\mathfrak e_1} \ar[d]^\mu \ar[r] & \mathbb C \ar@{=}[d] \ar[r] & 0 \\
0 \ar[r] & \mathfrak B_2 \ar[r] & \mathfrak D_{\mathfrak e_2} \ar[r] & \mathbb C \ar[r] & 0
}
\]
commute.
\end{lemma}
\begin{proof}
The Busby maps of the extensions in the above diagram are $\mathbb C \ni \lambda \mapsto \lambda \tau_i(1_{\mathfrak A_i})$ so the result follows immediately from \cite[Theorem 2.2]{EilersLoringPedersen-Busby}.
\end{proof}

When considering the ordered $K$-theory $K_\ast^+(\mathfrak A) = (K_0^+(\mathfrak A), K_1(\mathfrak A))$ for \emph{unital} $C^\ast$-algebras $\mathfrak A$, we will often add the class of the unit to the invariant
\[
K_\ast^{+,u}(\mathfrak A) := (K_0^+(\mathfrak A), [1_{\mathfrak A}]_0, K_1(\mathfrak A)).
\]
Alternatively, we may consider the unital embedding $j \colon \mathbb C \hookrightarrow \mathfrak A$. This gives an induced diagram
\begin{equation}\label{eq:unit}
j_\ast \colon K_\ast^{+}(\mathbb C) \to K_\ast^{+}(\mathfrak A).
\end{equation}
This diagram contains exactly the same information as $K_\ast^{+,u}(\mathfrak A)$, thus motivating the following construction. 

Suppose $\mathfrak e  : 0 \to \mathfrak B \xrightarrow \iota \mathfrak E \xrightarrow \pi \mathfrak A \to 0$ is an extension of $C^\ast$-algebras for which $\mathfrak A$ is unital, but where $\mathfrak E$ is not necessarily unital. We assume for convenience that $1_{\mathfrak A}$ lifts to a projection in $\mathfrak E$.

Again, we have a unital embedding $j \colon \mathbb C \hookrightarrow \mathfrak A$, and we obtain the following pull-back diagram
\[
\xymatrix{
0 \ar[r] & \mathfrak{B} \ar@{=}[d] \ar[r] & \mathfrak D_{\mathfrak e} \ar@{^{(}->}[d] \ar[r] & \mathbb C \ar[d]^{j} \ar[r] & 0 \\
0 \ar[r] & \mathfrak B \ar[r]^\iota & \mathfrak E \ar[r]^{\pi} & \mathfrak A\ar[r] & 0,
}
\]
where $\mathfrak D_{\mathfrak e}$ is as in Notation \ref{n:De}. Our invariant will be to apply $K$-theory with order and scale to this diagram, thus obtaining the following commutative diagram
\[
\xymatrix{
K_0^{+,\Sigma}(\mathfrak B) \ar@{=}[d] \ar[r] & K_0^{+,\Sigma}(\mathfrak D_{\mathfrak e}) \ar[d] \ar[r] & K_0^{+,\Sigma}(\mathbb C) \ar[d]^{j_0} \\
K_0^{+,\Sigma}(\mathfrak B) \ar[r]^{\iota_0} & K_0^{+,\Sigma}(\mathfrak E) \ar[r]^{\pi_0} & K_0^{+,\Sigma}(\mathfrak A) \ar[d]^{\delta_0} \\
K_1(\mathfrak A) \ar[u]^{\delta_1} & K_1(\mathfrak E) \ar[l]_{\pi_1} & K_1(\mathfrak B). \ar[l]_{\iota_1}
}
\]
We denoted this diagram by $\Knew( \mathfrak{e} )$. Homomorphisms between such diagrams are defined in the obvious way.

Suppose that $\mathfrak e_i \colon 0 \to \mathfrak B_i \to \mathfrak E_i \to \mathfrak A_i \to 0$ are extensions of $C^\ast$-algebras for $i=1,2$ with $\mathfrak A_i$ unital. Suppose that there is a commutative diagram
\[
\xymatrix{
\mathfrak e_1 : & 0 \ar[r] & \mathfrak B_1 \ar[d]^\beta \ar[r] & \mathfrak E_1 \ar[d]^\eta \ar[r] & \mathfrak A_1 \ar[d]^\alpha \ar[r] & 0 \\
\mathfrak e_2 : & 0 \ar[r] & \mathfrak B_2 \ar[r] & \mathfrak E_2 \ar[r] & \mathfrak A_2 \ar[r] & 0
}
\]
where all maps are $\ast$-homomorphisms, and $\alpha$ is unital. Then $\eta(\mathfrak D_{\mathfrak e_1}) \subseteq \mathfrak D_{\mathfrak e_2}$ and thus it easily follows that $(\beta, \eta , \alpha)$ induces a homomorphism $\Knew(\mathfrak e_1) \to \Knew(\mathfrak e_2)$.

In the cases we will be considering below, we assume that $\mathfrak A$ is a unital UCT Kirchberg algebra, $\mathfrak B$ is a stable AF algebra, and $\mathfrak E$ contains a full, properly infinite projection. Hence the order and scale can be ignored in $K_0(\mathfrak E)$ and $K_0(\mathfrak A)$, and the scale of $K_0(\mathfrak B)$ can be ignored when considering $\Knew(\mathfrak e)$.

We obtain our final classification result which is exactly Theorem \ref{t:classnonunital}.

\begin{theorem}
Let $\mathfrak e_i : 0 \to \mathfrak B_i \to \mathfrak E_i \to \mathfrak A_i \to 0$ be full extensions of $C^\ast$-algebras for $i=1,2$, such that $\mathfrak A_1$ and $\mathfrak A_2$ are \emph{unital} UCT Kirchberg algebras, and $\mathfrak B_1$ and $\mathfrak B_2$ are stable AF algebras. Then $\mathfrak E_1 \cong \mathfrak E_2$ if and only if $\Knew(\mathfrak e_1) \cong \Knew(\mathfrak e_2)$.
\end{theorem}
\begin{proof}
Suppose $\mathfrak E_1 \cong \mathfrak E_2$. As the extension $\mathfrak e_i$ is full, as $\mathfrak A_i$ is simple and $\mathfrak B_i$ is stable, it follows that $\mathfrak B_i$ is the unique maximal ideal in $\mathfrak E_i$ for $i=1,2$ (see Footnote \ref{fn}). It follows that the extensions $\mathfrak e_1$ and $\mathfrak e_2$ are isomorphic, and thus $\Knew(\mathfrak e_1) \cong \Knew(\mathfrak e_2)$.

For the converse, suppose that $\Knew(\mathfrak e_1) \cong \Knew(\mathfrak e_2)$, and let
\begin{eqnarray*}
\phi_\ast \colon K_\ast^{+,\Sigma}(\mathfrak A_1) \xrightarrow \cong K_\ast^{+,\Sigma}(\mathfrak A_2), && \psi_\ast \colon K_\ast^{+,\Sigma}(\mathfrak B_1) \xrightarrow \cong K_\ast^{+,\Sigma}(\mathfrak B_2), \\ 
\rho_\ast \colon K_\ast^{+,\Sigma}(\mathfrak E_1) \xrightarrow \cong K_\ast^{+,\Sigma}(\mathfrak E_2), && \theta_0 \colon K_0^{+,\Sigma}(\mathfrak D_{\mathfrak e_1}) \xrightarrow \cong K_0^{+,\Sigma}(\mathfrak D_{\mathfrak e_2})
\end{eqnarray*}
be a collection of isomorphisms inducing the isomorphism on $\Knew$. We first show that we may assume that $\mathfrak{A} = \mathfrak{A}_1 = \mathfrak{A}_2$, $\mathfrak{B} = \mathfrak{B}_1 = \mathfrak{B}_2$, $\phi_\ast = \id_{K_\ast(\mathfrak A)}$, $\psi_\ast = \id_{K_\ast(\mathfrak B)}$, that $\tau_1(1_{\mathfrak A}) = \tau_2(1_{\mathfrak A})$, where $\tau_i$ is the Busby map of $\mathfrak e_i$ for $i=1,2$, and that $\theta_0 = K_0(\mu)$ where $\mu \colon \mathfrak D_{\mathfrak e_1} \to \mathfrak D_{\mathfrak e_2}$ is the isomorphism provided by Lemma \ref{l:De}.

By the Kirchberg--Phillips theorem \cite{Kirchberg-simple}, \cite{Phillips-classification} we may pick an isomorphism $\alpha \colon \mathfrak{A}_1 \xrightarrow \cong \mathfrak{A}_2$ such that $K_\ast(\alpha) = \phi_\ast$. 

As $\mathfrak D_{\mathfrak e_i}$ is an extension of two AF algebras, it is itself an AF algebra by \cite[Chapter 9]{Effros-book-dim}. Hence by Elliott's classification of AF algebras \cite{Elliott-AFclass} we may pick an isomorphism $\mu \colon \mathfrak D_{\mathfrak e_1} \xrightarrow \cong \mathfrak D_{\mathfrak e_2}$ such that $K_0(\mu) = \theta_0$. In particular, $\mu$ restricts to an isomorphism $\beta \colon \mathfrak B_1 \xrightarrow \cong \mathfrak B_2$ satisfying $K_0(\beta) = \psi_0$.

Forming the push-out extension $\beta \cdot \mathfrak e_1$ and the pull-back extension $\mathfrak e_2 \cdot \alpha$, we obtain a diagram identical to \eqref{eq:bigpbpo}. By Lemma \ref{l:De} we get
\[
\overline \beta \circ \tau_1(1_{\mathfrak A_1}) = \tau_2(1_{\mathfrak A_2}) = \tau_2 \circ \alpha(1_{\mathfrak A_1}).
\]
Let 
\[
\mu^{(1)} \colon \mathfrak D_{\mathfrak e_1} \xrightarrow \cong \mathfrak D_{\beta \cdot \mathfrak e_1}, \qquad \mu^{(2)} \colon \mathfrak D_{\mathfrak e_2 \cdot \alpha} \xrightarrow \cong \mathfrak D_{\mathfrak e_2}
\]
be the induced isomorphisms, i.e.~the restriction--corestriction of $\eta^{(1)}$ and $\eta^{(2)}$ respectively.

Now, it follows from \eqref{eq:bigpbpo} (by inverting the isomorphisms) that we obtain induced isomorphisms $\Knew(\beta \cdot \mathfrak e_1) \xrightarrow \cong \Knew(\mathfrak e_1)$ and $\Knew(\mathfrak e_2) \xrightarrow \cong \Knew(\mathfrak e_2 \cdot \alpha)$. By composing these isomorphisms with the already given isomorphism $\Knew(\mathfrak e_1) \xrightarrow \cong \Knew (\mathfrak e_2)$, it follows that the compositions
\begin{eqnarray*}
K_\ast(\alpha)^{-1} \circ \phi_\ast \circ K_\ast(\id_{\mathfrak A_1})^{-1} = \id_{K_\ast(\mathfrak A_1)}, && K_\ast(\id_{\mathfrak B_2})^{-1} \circ \psi_\ast \circ K_\ast(\beta)^{-1} = \id_{K_\ast(\mathfrak B_2)} \\
K_\ast(\eta^{(2)})^{-1} \circ \rho_\ast \circ K_\ast(\eta^{(1)})^{-1} , && K_0(\mu^{(2)})^{-1} \circ \theta_0 \circ K_0(\mu^{(1)})^{-1}
\end{eqnarray*}
give rise to an isomorphism $\Knew(\beta \cdot \mathfrak e_1) \xrightarrow \cong \Knew(\mathfrak e_2 \cdot \alpha)$. Moreover, observe that $\mu^{(0)} := (\mu^{(2)})^{-1} \circ \mu \circ (\mu^{(1)})^{-1}$ is the unique (by Lemma \ref{l:De}) $\ast$-homomorphism making the diagram
\[
\xymatrix{
0 \ar[r] & \mathfrak B_2 \ar@{=}[d] \ar[r] & \mathfrak D_{\beta \cdot \mathfrak e_1} \ar[d]^{\mu^{(0)}} \ar[r] & \mathbb C \ar@{=}[d] \ar[r] & 0 \\
0 \ar[r] & \mathfrak B_2 \ar[r] & \mathfrak D_{\mathfrak e_2 \cdot \alpha} \ar[r] & \mathbb C \ar[r] & 0
}
\]
commute, and that $K_0(\mu^{(0)}) = K_0(\mu^{(2)})^{-1} \circ \theta_0 \circ K_0(\mu^{(1)})^{-1}$.

Therefore, without loss of generality, we may assume that $\mathfrak{A} = \mathfrak{A}_1 = \mathfrak{A}_2$, $\mathfrak{B} = \mathfrak{B}_1 = \mathfrak{B}_2$, $\phi_\ast = \id_{K_\ast(\mathfrak A)}$, $\psi_\ast = \id_{K_\ast(\mathfrak B)}$ that $\tau_1(1_{\mathfrak A}) = \tau_2(1_{\mathfrak A})$,  and that that $\theta_0 = K_0(\mu)$ where $\mu \colon \mathfrak D_{\mathfrak e_1} \to \mathfrak D_{\mathfrak e_2}$ is the map provided by Lemma \ref{l:De} (with $\beta = \id_{\mathfrak B}$).

As $\mathfrak B$ has real rank zero and $K_1(\mathfrak B) = 0$, the projection $\tau_1(1_{\mathfrak A}) \in \corona{\mathfrak B}$ lifts to a projection $p \in \multialg{\mathfrak B}$ by \cite[Corollary 3.16]{BrownPedersen-rr0}. In particular, by identifying $\mathfrak E_i$ with $\mathfrak A \oplus_{\tau_i , \pi_\mathfrak{B}} \multialg{\mathfrak B}$ in the canonical way, $1_{\mathfrak A} \oplus p$ defines a projection both in $\mathfrak E_1$ and in $\mathfrak E_2$ since $\tau_1(1_{\mathfrak A}) = \tau_2(1_{\mathfrak A})$. Note that when identifying $\mathfrak D_{\mathfrak e_1}$ and $\mathfrak D_{\mathfrak e_2}$ in a canonical way with a subalgebra of $\mathfrak A \oplus \multialg{\mathfrak B}$, then $\mu$ is simply the identity map. Hence $\mu(1_{\mathfrak A} \oplus p) = 1_{\mathfrak A} \oplus p$. In particular, by commutativity of the diagram
\[
\xymatrix{
K_0(\mathfrak D_{\mathfrak e_1}) \ar[r] \ar[d]^{\theta_0 = K_0(\mu)} & K_0(\mathfrak E_1) \ar[d]^{\rho_0} \\
K_0(\mathfrak D_{\mathfrak e_2}) \ar[r] & K_0(\mathfrak E_2),
}
\]
which is part of $\Knew(\mathfrak e_1) \to \Knew(\mathfrak e_2)$, it follows that $\rho_0([1_{\mathfrak A} \oplus p]) = [1_{\mathfrak A} \oplus p]$.

By Theorem \ref{t:fullext} it follows that $1_{\mathfrak A} \oplus p$ is a full, properly infinite projection in both $\mathfrak E_1$ and $\mathfrak E_2$. Moreover, $p\mathfrak B p$ is a full and stable corner in $\mathfrak B$ by Lemma \ref{l:stableAF}. Let 
\[
\iota \colon p \mathfrak B p \hookrightarrow \mathfrak B , \qquad \iota_i \colon (1_\mathfrak{A} \oplus p) \mathfrak E_i (1_{\mathfrak A} \oplus p) \hookrightarrow \mathfrak E_i
\]
for $i=1,2$ denote the inclusions, which are all inclusions of full, hereditary, $C^\ast$-subalgebras in separable $C^\ast$-algebras and thus induce isomorphisms in $K$-theory. Since $\rho_0([1_{\mathfrak A} \oplus p]) = [1_{\mathfrak A} \oplus p]$ it follows that the map
\[
K_\ast(\iota_2)^{-1} \circ \rho_\ast \circ K_\ast(\iota_1) \colon K_\ast((1_{\mathfrak A} \oplus p)\mathfrak E_1(1_{\mathfrak A} \oplus p)) \to K_\ast((1_{\mathfrak A} \oplus p)\mathfrak E_2(1_{\mathfrak A} \oplus p)) 
\]
induces a congruence $\Ksix(p \mathfrak e_1 p) \equiv \Ksix(p \mathfrak e_2 p)$, where $p\mathfrak e_i p$ denotes the unital extension
\[
0 \to p\mathfrak B p \to (1_{\mathfrak A} \oplus p)\mathfrak E_1(1_{\mathfrak A} \oplus p) \to \mathfrak A \to 0
\]
for $i=1,2$.

Thus, by Proposition \ref{p:congtoExt} there is an automorphism $\alpha \in \Aut(\mathfrak A)$ such that $p \mathfrak e_1 p$ and $p \mathfrak e_2 p \cdot \alpha = p (\mathfrak e_2 \cdot \alpha) p$ are strongly unitarily equivalent. By Lemma \ref{l:unitvsnonunit} it follows that $\mathfrak e_1$ and $\mathfrak e_2 \cdot \alpha$ are strongly unitarily equivalent. As the extension algebra of $\mathfrak e_2 \cdot \alpha$ is isomorphic to $\mathfrak E_2$, it follows that $\mathfrak E_1 \cong \mathfrak E_2$ as desired.
\end{proof}

\begin{remark}
In a future paper \cite{EGKRT-extgraph} we compute the range of the invariant $\Knew$ for graph $C^\ast$-algebras with a unique, non-trivial ideal. This will be used to show that an extension of two simple graph $C^\ast$-algebras is again a graph $C^\ast$-algebra, provided there are no $K$-theoretic obstructions.
\end{remark}

\newcommand{\etalchar}[1]{$^{#1}$}

\end{document}